\documentclass[a4paper,11pt, leqno]{amsart}

\usepackage{amsmath, amsthm, amsfonts, amssymb} 
\newtheorem{theo}{Theorem}[section]

\newtheorem*{conv}{Important Convention}

\newtheorem*{ack}{Acknowledgment}

\newtheorem{df}[theo]{Definition}

\newtheorem{term}[theo]{Terminology}

\newtheorem{cor}[theo]{Corollary}

\newtheorem{lem}[theo]{Lemma}

\newtheorem{nota}[theo]{Notation}

\newtheorem{prop}[theo]{Proposition}

\newtheorem{rem}[theo]{Remark}

\newtheorem{exam}[theo]{Example}

\newtheorem{claim}[theo]{Claim}

\newcommand{\fonction}[4]
{\begin{array}{rcl}
#1 & \to & #2 \\
#3 & \mapsto & #4 \\
\end{array}}

\newcommand{\R}{\mathbf{R}}
\newcommand{\C}{\mathbf{C}}
\newcommand{\Z}{\mathbf{Z}}
\newcommand{\F}{\mathbf{F}}

\newcommand{\N}{\mathbf{N}}
\newcommand{\Sd}{\operatorname{Sd}}
\newcommand{\Sp}{\operatorname{Sp}}

\newcommand{\Id}{\operatorname{Id}}
\newcommand{\Ad}{\operatorname{Ad}}
\newcommand{\id}{\operatorname{id}}
\newcommand{\Tr}{\operatorname{Tr}}

\newcommand{\Aut}{\operatorname{Aut}}
\newcommand{\SL}{\operatorname{SL}}
\newcommand{\alt}{\operatorname{alt}}
\newcommand{\tr}{\operatorname{tr}}
\newcommand{\alg}{\operatorname{alg}}
\newcommand{\range}{\operatorname{range}}
\renewcommand{\span}{\operatorname{span}}

\begin{document}

\title[Type ${\rm II_1}$ Factors with Prescribed Fundamental Group]{Construction of Type ${\rm II_1}$ Factors with Prescribed Countable Fundamental Group}

\author{Cyril Houdayer}

\address{UCLA \\
                 Department of Mathematics \\
                 Los Angeles \\
                 CA 90095 \\
                 USA}

\email{cyril@math.ucla.edu}



\begin{abstract}
In the context of Free Probability Theory, we study two different constructions that provide new examples of factors of type ${\rm II_1}$ with prescribed countable fundamental group. First we investigate state-preserving group actions on the almost periodic free Araki-Woods factors satisfying both a condition of mixing and a condition of free malleability in the sense of Popa. Typical examples are given by the free Bogoliubov shifts. Take an ICC $w$-rigid group $G$ such that $\mathcal{F}(L(G)) = \{1\}$ (e.g. $G = \Z^2 \rtimes \SL(2, \Z)$). For any countable subgroup $S \subset \R^*_+$, we construct an action of $G$ on $L(\F_\infty)$ such that the associated crossed product $L(\F_\infty) \rtimes G$ is a type ${\rm II_1}$ factor and its fundamental group is $S$. The second construction is based on a free product. Take $(B(H), \psi)$ any factor of type ${\rm I}$ endowed with a faithful normal state and denote by $S \subset \R^*_+$ the subgroup generated by the point spectrum of $\psi$. We show that the centralizer $(L(G) \ast B(H))^{\tau \ast \psi}$ is a type ${\rm II_1}$ factor and its fundamental group is $S$. Our proofs rely on Popa's deformation/rigidity strategy using his intertwining-by-bimodules technique.
\end{abstract}

\maketitle

\section{Introduction}

Popa introduced in \cite{{popamal1}, {popamal2}, {popamsri}} the following remarkable concept: a state-preserving action $\sigma$ of a group $G$ on a von Neumann algebra $(\mathcal{N}, \varphi)$ is said to be {\it malleable} if there exists a continuous action $\alpha : \R \to \Aut(\mathcal{N} \otimes \mathcal{N}, \varphi \otimes \varphi)$ which commutes with the diagonal action $(\sigma_g \otimes \sigma_g)$ and such that $\alpha_1(a \otimes 1) = 1 \otimes a$, for any $a \in \mathcal{N}$. It is said to be {\it s-malleable} if there moreover exists an automorphism $\beta$ of $(\mathcal{N} \otimes \mathcal{N}, \varphi \otimes \varphi)$ commuting with $(\sigma_g \otimes \sigma_g)$ such that $\beta\alpha_t = \alpha_{-t}\beta$ for all $t \in \R$ and $\beta(a \otimes 1) = a \otimes 1$ for all $a \in \mathcal{N}$ and such that $\beta$ has period $2$: $\beta^2 = \id$. Typical examples of such actions are given by the commutative Bernoulli shifts. The remarkable idea of Popa was to combine this \emph{deformation} property of the action with a \emph{rigidity} property, namely the \emph{relative property} $(T)$ of Kazhdan and Margulis (\cite{{kazhdan}, {margulis}}) of the group. Using this {\textquotedblleft tension\textquotedblright} between deformation and rigidity, Popa solved  longstanding problems in the theory of von Neumann algebras. We refer to the papers of Popa and his coauthors \cite{{ipp}, {popasup}, {popamal1}, {popamal2}, {popa2001}, {popamsri}, {popavaes}}, and to Vaes' Bourbaki Seminar \cite{vaesbern} for the stunning applications of these \emph{deformation/rigidity} phenomena.

In the context of Free Probability Theory, Popa introduced another notion of \emph{malleability} where the free product naturally replaces the tensor product. 

\begin{df}[Popa, \cite{popamsri}]\label{malleable}
\emph{Let $G$ be a countable discrete group. Let $\sigma$ be a state-preserving action of $G$ on the von Neumann algebra $(\mathcal{N}, \varphi)$.}
\begin{enumerate}
\item [$\bullet$] \emph{The action is said to be \emph{malleable} if there exists a continuous action $\alpha : \R \to \Aut(\mathcal{N} \ast \mathcal{N}, \varphi \ast \varphi)$ which commutes with the diagonal action $(\sigma_g \ast \sigma_g)$ and such that $\alpha_1(a \ast 1) = 1 \ast a$, $\forall a \in \mathcal{N}$.}
\item [$\bullet$] \emph{It is said to be \emph{s-malleable} if there moreover exists a period $2$ automorphism $\beta$ of $(\mathcal{N} \ast \mathcal{N}, \varphi \ast \varphi)$ commuting with $(\sigma_g \ast \sigma_g)$ such that $\beta\alpha_t = \alpha_{-t}\beta$ for all $t \in \R$ and $\mathcal{N} \ast \C \subset (\mathcal{N} \ast \mathcal{N})^\beta$.} 
\end{enumerate}
\end{df}
\begin{conv}
\emph{In the rest of this paper, the notion of \emph{malleability} or \emph{s-malleability} will always be taken in the sense of Definition $\ref{malleable}$, i.e. in the {\textquotedblleft free\textquotedblright} sense.
}
\end{conv}

We introduce the following important notation: 

\begin{nota}
\emph{Let $(\mathcal{N}, \varphi)$ be a von Neumann algebra endowed with a faithful normal state, $\mathcal{N}^\varphi$ denotes the centralizer of the state $\varphi$. We set $\mathcal{N} \ominus \C = \mathcal{N} \cap \ker(\varphi)$. More generally, let $\mathcal{B}\subset \mathcal{N}$ be a von Neumann subalgebra globally invariant under the modular group $(\sigma_t^\varphi)$; if $E_{\mathcal{B}} : \mathcal{N} \to \mathcal{B}$ denotes the unique state-preserving conditional expectation, we set $\mathcal{N} \ominus \mathcal{B} = \mathcal{N} \cap \ker E_\mathcal{B}$.
}
\end{nota}
In the context of free probability, we present, for a state-preserving group action, a stronger mixing property than the usual one.
\begin{df}
\emph{Let $G$ be a countable discrete group. Let $\sigma$ be a state-preserving action on the von Neumann algebra $(\mathcal{N}, \varphi)$. The action is said to be \emph{freely mixing} if for all $n \in \N^*$, $x_1, \dots, x_n, y_0, y_1, \dots, y_n \in \mathcal{N} \ominus \C$, except possibly $y_0$ and/or $y_n$ are equal to $1$, we have}
\begin{equation*}
\lim_{g_1, \dots, g_n \to \infty} \varphi(y_0\sigma_{g_1}(x_1)y_1 \cdots \sigma_{g_n}(x_n)y_n) = 0.
\end{equation*}
\end{df}
Obviously, a \emph{freely mixing} action is \emph{strongly mixing}. We shall show (see Section $\ref{exa}$) that free Bernoulli actions and free Bogoliubov shifts on the almost periodic free Araki-Woods factors are typical examples of (s-)malleable, freely mixing actions.

\begin{term}
\emph{A $w$-{\it rigid} group $G$ is a group that admits an infinite normal subgroup $H$, such that the pair $(G, H)$ has the relative property $(T)$ of Kazhdan and Margulis \cite{{kazhdan}, {margulis}}. The example par excellence of such a pair is $(G, H) = (\Z^2 \rtimes \SL(2, \Z), \Z^2)$. Other examples include $(\Z^2 \rtimes \Gamma, \Z^2)$ where $\Gamma$ is any nonamenable subgroup of $\SL(2, \Z)$ acting on $\Z^2$ by its given embedding in $\SL(2, \Z)$ (see \cite{{burger}, {shalom}}). Of course, any group $G$ with the property $(T)$ is $w$-rigid.
}
\end{term}

In this paper, we present two different constructions that produce new examples of type ${\rm II_1}$ factors with a prescribed countable fundamental group. Remind that for a type ${\rm II_1}$ factor $M$, the {\it fundamental group} of $M$ is defined as follows:
\begin{equation*}
\mathcal{F}(M) := \{\tau(p)/\tau(q) : pMp \simeq qMq\},
\end{equation*}
where $p, q$ are projections in $M$. The first construction is based on a \emph{crossed product}. In \cite{popamal1}, we remind that Popa proved several results of intertwining of rigid subalgebras in crossed products. Using the so-called Connes-St\o rmer Bernoulli shifts, he constructed actions of ICC $w$-rigid groups on the hyperfinite type ${\rm II_1}$ factor $\mathcal{R}$ such that the associated crossed products have prescribed countable fundamental group. We prove an analogue of this result with $L(\F_\infty)$ instead of $\mathcal{R}$. The type ${\rm II_1}$ factor $L(\F_\infty)$ appears naturally as the centralizer of the free quasi-free state for the almost periodic free Araki-Woods factors. We obtain the following result:

\begin{theo}\label{funda}
Let $S \subset \R^*_+$ be a countable subgroup. Let $G$ be an ICC $w$-rigid group. Assume that $\mathcal{F}(L(G)) = \{1\}$. Then, there exists an action of $G$ on the type ${\rm II_1}$ factor $L(\F_\infty)$ such that the crossed product $L(\F_\infty) \rtimes G$ is a type ${\rm II_1}$ factor with fundamental group equal to $S$.
\end{theo}

The typical example of a group $G$ satisfying the conditions of Theorem $\ref{funda}$, is $G = \Z^2 \rtimes \SL(2, \Z)$ \cite{popa2001}. Once again, the actions considered in Theorem $\ref{funda}$ are \emph{concrete}: they are the free Bogoliubov shifts.

The second construction is based on a \emph{free product}. In \cite{ipp}, among other remarkable results, Ioana, Peterson \& Popa gave several examples of type ${\rm II_1}$ factors with a prescribed fundamental group using the free product construction in the tracial case. We shall generalize some of their techniques to the \emph{almost periodic} case. For $(\mathcal{A}, \psi)$, a von Neumann algebra endowed with an almost periodic faithful normal state, denote by $\Sp(\mathcal{A}, \psi)$ the point spectrum of $\psi$ and by $\Gamma_{\Sp(\mathcal{A}, \psi)} \subset \R^*_+$ the subgroup generated by $\Sp(\mathcal{A}, \psi)$. We obtain the following (see Theorem \ref{generalresult} for a more general version):

\begin{theo}\label{funda2}
Let $G$ be an ICC $w$-rigid group such that $\mathcal{F}(L(G)) = \{1\}$. Let $(\mathcal{A}, \psi)$ be a  von Neumann algebra endowed with an almost periodic state such that the centralizer $\mathcal{A}^\psi$ has the Haagerup property. Write $M = (L(G) \ast \mathcal{A})^{\tau \ast \psi}$. Then $M$ is a type ${\rm II_1}$ factor and its fundamental group is $\Gamma_{\Sp(\mathcal{A}, \psi)}$.
\end{theo}

Many examples of such von Neumann algebra $\mathcal{A}$ do exist: amenable von Neumann algebras endowed with a faithful normal almost periodic state, almost periodic free Araki-Woods factors with their free quasi-free state, but also all the free products studied by Dykema in \cite{dykema96}.  

This paper is organized as follows. Section $\ref{ara}$ is devoted to a few preliminaries. In Section $\ref{exa}$, we give the main examples of s-malleable freely mixing actions on these factors. In Section $\ref{te}$, we show some technical results about the intertwining of subalgebras. At last, in Section $\ref{fun}$ we use the \emph{deformation/rigidity} strategy {\textquotedblleft \`a la\textquotedblright} Popa to prove Theorems $\ref{funda}$ and $\ref{funda2}$. The Appendix is devoted to prove some well known facts about the polar decomposition of a vector.

\begin{nota}
\emph{Throughout this paper, we shall use the following notation: $G$ is any countable discrete group  and $(\mathcal{N}, \varphi)$ is any von Neumann algebra endowed with a faithful normal almost periodic state \cite{connes74}. Any group action is always assumed to be state-preserving. Any von Neumann algebra $\mathcal{N}$ is always assumed to have separable predual. Any state, conditional expectation is assumed to be normal and faithful. If $(\mathcal{M}, \varphi), (\mathcal{N}, \psi)$ are von Neumann algebras endowed with states, whenever we write $(\mathcal{M}, \varphi) \cong (\mathcal{N}, \psi)$, we mean that there exists a $\ast$-isomorphism $\theta : \mathcal{M} \to \mathcal{N}$ such that $\psi \circ \theta = \varphi$. For $n \in \N^*$ and a von Neumann algebra $\mathcal{M}$, we set $\mathcal{M}^n = M_n(\C) \otimes  \mathcal{M}$. The canonical normalized trace on $M_n(\C)$ is usually denoted by $\tr_n$.} 
\end{nota}

We mention that recently, Popa \& Vaes \cite{popavaes2} proved the existence of free ergodic measure-preserving actions of $\mathbf{F}_\infty$ on the standard non-atomic probability space $(X, \mu)$ whose type ${\rm II_1}$ factors and orbit equivalence relations have prescribed fundamental group in a large class $S$ of subgroups of $\R^*_+$ that contains all countable subgroups and many uncountable subgroups. In particular, they obtained the first examples of separable type ${\rm II_1}$ factors and orbit equivalence relations with uncountable fundamental group different from $\R^*_+$.

\begin{ack}
\emph{The author would like to thank Stefaan Vaes for his constant support, but also for stimulating discussions on the subject and helpful suggestions regarding this manuscript.}
\end{ack}

\section{Preliminaries}\label{ara}

\subsection{Von Neumann Algebras Endowed with Almost Periodic States}

Most of the time, the von Neumann algebra $(\mathcal{M}, \varphi)$ will be assumed to have a faithful normal {\it almost periodic} state $\varphi$. We regard $\mathcal{M} \subset B(L^2(\mathcal{M}, \varphi))$ through the GNS construction. We denote by $\widehat{\cdot} : \mathcal{M} \to L^2(\mathcal{M}, \varphi)$ the canonical embedding. Let $S^0_\varphi$ be the  antilinear operator defined by
\begin{equation*}
S^0_\varphi : \fonction{\widehat{\mathcal{M}}}{L^2(\mathcal{M}, \varphi)}{\widehat{x}}{\widehat{x^*}.}
\end{equation*}
Thanks to Tomita-Takesaki theory, $S^0_\varphi$ is closable and we denote by $S_\varphi$ its closure. Write $S_\varphi = J_\varphi \Delta^{1/2}_\varphi$ for its polar decomposition. The \emph{modular automorphism  group} $(\sigma_t^\varphi)$ of $\mathcal{M}$ w.r.t. the state $\varphi$ is defined by $\sigma_t^\varphi = \Delta_\varphi^{it} \cdot \Delta_\varphi^{-it}$. Moreover, note that $L^2(\mathcal{M}, \varphi)$ comes naturally equipped with a structure of $\mathcal{M}$-$\mathcal{M}$ bimodule:
\begin{eqnarray*}
x \cdot \xi & := & x \xi, \\
\xi \cdot x & := & J_\varphi x^* J_\varphi \xi, \forall \xi \in L^2(\mathcal{M}, \varphi), \forall x \in \mathcal{M}. 
\end{eqnarray*}
In the sequel, we shall simply write $x\xi$ and $\xi x$ instead of $x \cdot \xi$ and $\xi \cdot x$.  

Denote by $\Sp(\mathcal{M}, \varphi)$ the point spectrum of the modular operator $\Delta_\varphi$. For $\gamma \in \Sp(\mathcal{M}, \varphi)$, denote by $\mathcal{M}^\gamma$ the vector subspace of $\mathcal{M}$ of all $\gamma$-\emph{eigenvectors} for $\varphi$, i.e.
\begin{eqnarray*}
\mathcal{M}^\gamma & = & \{ x \in \mathcal{M} : \sigma^\varphi_t(x) = \gamma^{it}x \}\\
& = & \{ x \in \mathcal{M} : \varphi(xy) = \gamma \varphi(yx), \forall y \in \mathcal{M}\}.
\end{eqnarray*}
Denote by $\mathcal{M}_{\alg} = \span \{\mathcal{M}^\gamma : \gamma \in \Sp(\mathcal{M}, \varphi)\}$, the linear span of the $\mathcal{M}^\gamma$'s. It is clear that $\mathcal{M}_{\alg}$ is a unital $\ast$-subalgebra of $\mathcal{M}$. Since the state $\varphi$ is assumed to be almost periodic, $\mathcal{M}_{\alg}$ is $\sigma$-weakly dense in $\mathcal{M}$. We can also write:
\begin{equation*}
L^2(\mathcal{M}, \varphi) = \bigoplus_{\gamma \in \Sp(\mathcal{M}, \varphi)} L^2(\mathcal{M}^\gamma).
\end{equation*}
It is straightforward to check that if $\gamma \in \Sp(\mathcal{M}, \varphi)$ and $\xi \in L^2(\mathcal{M}^\gamma)$, then for any $x \in \mathcal{M}^\lambda$, $x \xi$ and $\xi x \in L^2(\mathcal{M}^{\lambda \gamma})$.

The $L^2$-norm w.r.t. the state $\varphi$, simply denoted by $\| \cdot \|_2$, is defined as follows: $\|x\|_2 = \varphi(x^*x)^{1/2}$, for any $x \in \mathcal{M}$. Remind that the topology given by the norm $\|\cdot\|_2$ coincides with the strong topology on bounded sets of $\mathcal{M}$.

For $a \in \mathcal{M}$, set 
$$L_a : \fonction{L^2(\mathcal{M}, \varphi)}{L^2(\mathcal{M}, \varphi)}{\widehat{b}}{\widehat{ab}.}$$
The operator $L_a$ is always bounded and $\|L_a\| = \|a\|$. For $b \in \mathcal{M}$, set
$$R_b : \fonction{L^2(\mathcal{M}, \varphi)}{L^2(\mathcal{M}, \varphi)}{\widehat{a}}{\widehat{ab}.}$$
One must pay attention to the fact that the operator $R_b$ is unbounded in general. However, if $b \in \mathcal{M}^\gamma$ for some $\gamma \in \Sp(\mathcal{M}, \varphi)$, the Tomita-Takesaki theory claims that
\begin{equation*}
R_b = J_\varphi L_{\sigma^\varphi_{-i/2}(b^*)} J_\varphi = \gamma^{-1/2} J_\varphi L_{b^*} J_\varphi.
\end{equation*}
We refer for example to Lemma ${\rm VIII}. 3.18$ in \cite{takesakiII} for further details. Consequently, in this case, $R_b$ is bounded and $\|R_b\| = \gamma^{-1/2}\|b\|$. Thus, for any $x, z \in \mathcal{M}$, and for any $y \in \mathcal{M}^\gamma$, we have
\begin{eqnarray}\label{inequa}
\|x z y\|_2 & = & \|L_x R_y (\widehat{z})\|_2 \\ \nonumber
& = & \gamma^{-1/2} \|L_x J_\varphi L_{y^*} J_\varphi(\widehat{z})\|_2 \\ \nonumber
& \leq & \gamma^{-1/2} \|x\| \|y\| \|z\|_2.
\end{eqnarray}
We shall repeatedly use this inequality in the sequel.

\subsection{Hilbert (Bi)modules and the Basic Construction}

We reproduce here Appendix A of Vaes' Bourbaki seminar. Let $(\mathcal{N}, \varphi)$ be a von Neumann algebra endowed with an almost periodic state, and let $\mathcal{B} \subset \mathcal{N}$ be a von Neumann subalgebra globally invariant under the modular group $(\sigma_t^\varphi)$. We regard $\mathcal{N} \subset B(L^2(\mathcal{N}, \varphi))$ through the GNS construction. We shall simply denote $L^2(\mathcal{N}, \varphi)$ by $L^2(\mathcal{N})$. We denote by $E_{\mathcal{B}} : \mathcal{N} \to \mathcal{B}$ the unique state-preserving conditional expectation (see \cite{takesakiII} for further details). The \emph{basic construction} $\langle \mathcal{N}, e_{\mathcal{B}}\rangle$ is defined as the von Neumann subalgebra of $B(L^2(\mathcal{N}))$ generated by $\mathcal{N}$ and the orthogonal projection $e_{\mathcal{B}}$ of $L^2(\mathcal{N})$ onto $L^2(\mathcal{B}) \subset L^2(\mathcal{N})$. The relationship between $e_{\mathcal{B}}$ and $E_{\mathcal{B}}$ is as follows:
\begin{equation*}
e_{\mathcal{B}} x e_{\mathcal{B}} = E_{\mathcal{B}}(x) e_{\mathcal{B}}, \forall x \in \mathcal{N}.
\end{equation*}
It can be checked that $\langle \mathcal{N}, e_{\mathcal{B}}\rangle$ consists of those operators $T \in B(L^2(\mathcal{N}))$ that commute with the right module action of $\mathcal{B}$: $T(\xi b) = T(\xi) b$, $\forall \xi \in L^2(\mathcal{N}), \forall b \in \mathcal{B}$. In other words, 
\begin{equation*}
\langle \mathcal{N}, e_{\mathcal{B}}\rangle = (J_\varphi \mathcal{B} J_\varphi)' \cap B(L^2(\mathcal{N})).
\end{equation*}
The basic construction comes equipped with a canonical normal semifinite faithful weight $\widehat{\varphi}_{\mathcal{B}}$ which satisfies
\begin{equation*}
\widehat{\varphi}_{\mathcal{B}}(x e_{\mathcal{B}} y) = \varphi(xy), \forall x, y \in \mathcal{N}.
\end{equation*}
It can be checked that $\widehat{\varphi}_{\mathcal{B}}$ is also almost periodic and $\Sp(\langle \mathcal{N}, e_{\mathcal{B}}\rangle, \widehat{\varphi}_{\mathcal{B}}) = \Sp(\mathcal{N}, \varphi)$.

Let $(B, \tau)$ be any finite von Neumann algebra endowed with a faithful normal trace $\tau$. Let $\mathcal{K}$ be a right $B$-module. Denote by $B'$ the commutant of $B$ on $\mathcal{K}$, i.e. $B'$ consists of the operators $T \in B(\mathcal{K})$ that commute with the right module action of $B$. One can construct a faithful semifinite normal positive linear map 
\begin{equation*}
E' : (B')_+ \to \{\mbox{positive self-adjoint operators affiliated with }\mathcal{Z}(B)\}
\end{equation*}
satisfying $E'(x^*x) = E'(xx^*)$, $\forall x \in B'$. Moreover, whenever $T : L^2(B) \to K$ is bounded and right $B$-linear, $T^*T \in B$, $TT^* \in B'$, and we have
\begin{equation*}
E'(TT^*) = E_{\mathcal{Z}(B)}(T^*T).
\end{equation*}
The positive self-adjoint operator $E'(1_{B'})$ affiliated with the center $\mathcal{Z}(B)$ provides a complete invariant for right $B$-modules. It should be noted here that the right $B$-module $\mathcal{K}$ is finitely generated, i.e. of the form $pL^2(B)^{\oplus n}$ for some projection $p \in M_n(\C) \otimes B$, iff $E'(1_{B'})$ is bounded. In that case, $E'(1_{B'}) = (\tr_n \otimes E_{\mathcal{Z}(B)})(p)$. Write $\Tr = \tau \circ E'$. It follows that $\Tr$ is a faithful semifinite normal trace on $B'$. If $\Tr(1_{B'}) < \infty$, we shall say that the right $B$-module $\mathcal{K}$ is of \emph{finite trace over} $B$. Thus, $E'(1_{B'})$ is not bounded \emph{a priori} but is $\tau$-integrable. This implies that $E'(1_{B'})z$ is bounded for projections $z \in \mathcal{Z}(B)$ with trace arbitrary close to $1$. So, we have the following lemma:
\begin{lem}\label{finiment}
Let $\mathcal{K}$ be a right $B$-module of finite trace over $B$. Then, for any $\varepsilon > 0$, there exists a projection $z \in \mathcal{Z}(B)$ with $\tau (1 - z) \leq \varepsilon$, and such that the right $B$-module $\mathcal{K}z$ is finitely generated over $B$.
\end{lem}

Let's come back to the basic construction for the inclusion $B \subset \mathcal{N}$, and assume now that $B \subset \mathcal{N}^\varphi$. We observe that the restriction of $\varphi$ to $B$ defines a tracial state and  $\langle \mathcal{N}, e_B \rangle$ is precisely the commutant of $B$ on the right $B$-module $L^2(\mathcal{N})$. Using the previous paragraph, $\langle \mathcal{N}, e_B \rangle$ comes equipped with a canonical faithful semifinite normal trace $\Tr$. For $\gamma \in \Sp(\mathcal{N}, \varphi)$, denote by $p_\gamma$ the orthogonal projection of $L^2(\mathcal{N})$ onto $L^2(\mathcal{N}^\gamma)$. Let $\lambda \in \Sp(\mathcal{N}, \varphi)$ and $\xi \in L^2(\mathcal{N}^\lambda)$. Since $B \subset \mathcal{N}^\varphi$, one has $\xi b \in L^2(\mathcal{N}^\lambda)$, and so $p_\gamma (\xi b) = \delta_{\gamma, \lambda} \xi b = p(\xi) b$. Consequently, $p_\gamma \in \langle \mathcal{N}, e_B \rangle$. One can check that 
\begin{eqnarray*}
\widehat{\varphi}_B (x) & = & \sum_{\gamma \in \Sp(\mathcal{N}, \varphi)} \widehat{\varphi}_B(p_\gamma x p_\gamma) \\
 \Tr(x)  & = &  \sum_{\gamma \in \Sp(\mathcal{N}, \varphi)} \gamma^{-1} \widehat{\varphi}_B(p_\gamma x p_\gamma), \forall x \in \langle \mathcal{N}, e_B \rangle_+.
\end{eqnarray*}
In particular, on $p_\gamma \langle \mathcal{N}, e_B \rangle p_\gamma$, $\widehat{\varphi}_B$ is tracial and a multiple of $\Tr$, $\forall \gamma \in \Sp(\mathcal{N}, \varphi)$. In Section $\ref{te}$, we shall encounter the following situation. Let $p \in p_\gamma \langle \mathcal{N}, e_B \rangle p_\gamma$ be a projection such that $\widehat{\varphi}_B(p) < \infty$. Let $\mathcal{K}$ be the right $B$-module defined by $\mathcal{K} = pL^2(\mathcal{N})$. We can check that the commutant of $B$ on $\mathcal{K}$ is exactly $p \langle \mathcal{N}, e_B \rangle p$. Since $\Tr(p) = \gamma^{-1} \widehat{\varphi}_B(p) < \infty$, it follows that as a right $B$-module, $\mathcal{K}$ is of finite trace over $B$. Thus, one can apply Lemma $\ref{finiment}$.

\subsection{Free Araki-Woods Factors of Shlyakhtenko}

Let $H_{\R}$ be a real Hilbert space and let $(U_t)$ be an orthogonal representation of $\R$ on $H_{\R}$. Let $H_\C = H_{\R} \otimes_{\R} \C$ be the complexified Hilbert space. Let $A$ be the infinitesimal generator of $(U_t)$ on $H_\C$. Define another inner product on $H_\C$ by
\begin{equation*}
\langle \xi, \eta\rangle_{U} = \left\langle \frac{2}{1 + A^{-1}}\xi, \eta\right\rangle.
\end{equation*}
Note that for any $\xi \in H_\R$, $\|\xi\|_U = \|\xi\|$; also, for any $\xi , \eta \in H_\R$, $\Re(\langle \xi, \eta\rangle_U) = \langle \xi, \eta\rangle$. We refer to \cite{shlya97} for the main properties of this inner product. Denote by $H$ the completion of $H_\C$ for this new inner product. Introduce now the \emph{full Fock space} of $H$:
$$\mathcal{F}(H) =\C\Omega \oplus \bigoplus_{n = 1}^{\infty} H^{\otimes n}.$$ 
The unit vector $\Omega$ is called \emph{vacuum vector}. For any $\xi \in H$, we have the \emph{left creation operator}
$$l(\xi) : \mathcal{F}(H) \to \mathcal{F}(H) : \left\{ 
{\begin{array}{l} l(\xi)\Omega = \xi, \\ l(\xi)(\xi_1 \otimes \cdots \otimes \xi_n) = \xi \otimes \xi_1 \otimes \cdots \otimes \xi_n.
\end{array}} \right.
$$
For any $\xi \in H$, we denote by $s(\xi)$ the real part of $l(\xi)$ given by
$$s(\xi) = \frac{l(\xi) + l(\xi)^*}{2}.$$
The crucial result of Voiculescu \cite{voiculescu92} claims that the distribution of the operator $s(\xi)$ w.r.t. the vacuum vector state $\varphi_U(x) = \langle x\Omega, \Omega\rangle_U$ is the semicircular law of Wigner supported on the interval $[-\|\xi\|, \|\xi\|]$. 

\begin{df}[Shlyakhtenko, \cite{shlya97}]
\emph{Let $(U_t)$ be an orthogonal representation of $\R$ on the real Hilbert space $H_{\R}$ $(\dim H_{\R} \geq 2)$. The \emph{free Araki-Woods factor} associated with $H_\R$ and $(U_t)$, denoted by $\Gamma(H_{\R}, U_t)''$, is defined by}
$$\Gamma(H_{\R}, U_t)'' = \{s(\xi) : \xi \in H_{\R}\}''.$$
\emph{The vector state $\varphi_{U}(x) = \langle x\Omega, \Omega\rangle_U$ is called the \emph{free quasi-free state}.}
\end{df}

The free Araki-Woods factors provided many new examples of full factors of type {\rm III} \cite{{barnett95}, {connes73}, {shlya2004}}. We can summarize some of their general properties in the following theorem (see also Vaes' Bourbaki seminar \cite{vaes2004}):

\begin{theo}[Shlyakhtenko, \cite{{shlya2004}, {shlya99}, {shlya98}, {shlya97}}]
Let $(U_t)$ be an orthogonal representation of $\R$ on the real Hilbert space $H_{\R}$ with $\dim H_{\R} \geq 2$. Denote by $\mathcal{N} = \Gamma(H_{\R}, U_t)''$.
\begin{enumerate}
\item $\mathcal{N}$ is of type ${\rm II_1}$ iff $U_t = id$ for every $t \in \R$.
\item $\mathcal{N}$ is of type ${\rm III_{\lambda}}$ $(0 < \lambda < 1)$ iff $(U_t)$ is periodic of period $\frac{2\pi}{|\log \lambda|}$.
\item $\mathcal{N}$ is of type ${\rm III_1}$ in the other cases.
\item If $(U_t)$ is almost periodic, then $\varphi_U$ is an almost periodic state.
\end{enumerate}
\end{theo}

\begin{rem}[\cite{shlya97}]
\emph{Explicitly the value of $\varphi_U$ on a word in $s(h_\iota)$ is given by}
\begin{equation}\label{formula}
\varphi_U(s(h_1) \cdots s(h_n)) = 2^{-n}\sum_
{(\{\beta_i, \gamma_i\}) 
 \in NC(n), 
\beta_i < \gamma_i}
\prod_{k = 1}^{n/2}\langle h_{\beta_k}, h_{\gamma_k}\rangle_U.
\end{equation}
\emph{for $n$ even and is zero otherwise. Here $NC(2p)$ stands for all the noncrossing pairings of $\{1, \dots, 2p\}$, i.e. pairings for which whenever $a < b < c < d$, and $a, c$ are in the same class, then $b, d$ are not in the same class. The total number of such pairings is given by the $p$-th Catalan number}
\begin{equation*}
C_p = \frac{1}{p + 1}\begin{pmatrix}
2p \\
p
\end{pmatrix}.
\end{equation*}
\end{rem}

In the almost periodic case, using a powerful tool called the \emph{matricial model}, Shlyakhtenko obtained the following remarkable result:

\begin{theo}[Shlyakhtenko, \cite{{shlya2004}, {shlya97}}]\label{pparaki}
Let $(U_t)$ be a nontrivial almost periodic orthogonal representation of $\R$ on the real Hilbert space $H_{\R}$ with $\dim H_{\R} \geq 2$. Let $A$ be the infinitesimal generator of $(U_t)$ on $H_{\C}$, the  complexified Hilbert space of $H_{\R}$. Denote by $\mathcal{N} = \Gamma(H_{\R}, U_t)''$. Let $\Gamma \subset \R^*_+$ be the subgroup generated by the point spectrum of $A$. Then, $\mathcal{N}$ only depends on $\Gamma$ up to state-preserving isomorphisms. 

Conversely, the group $\Gamma$ coincides with the $\Sd$ invariant of the factor $\mathcal{N}$. Consequently, $\Sd$ completely classifies the almost periodic free Araki-Woods factors. Moreover, the centralizer of the free quasi-free state $\varphi_U$ is isomorphic to the type ${\rm II_1}$ factor $L(\F_{\infty})$.
\end{theo}

Let $K_{\R}$ be an infinite dimensional separable real Hilbert space, and let $0 < \lambda < 1$. We define on $K_\R \oplus K_\R$ the following one-parameter family of orthogonal transformations:
\begin{equation}\label{periodic}
U^\lambda_t = \begin{pmatrix}
\cos(t\log \lambda) & - \sin(t\log \lambda) \\
\sin(t\log \lambda) & \cos(t\log \lambda)
\end{pmatrix}.
\end{equation}
\begin{nota}\label{Tlambda}
\emph{Write $(T_{\lambda}, \varphi_{\lambda}) := \Gamma(H_{\R}, U^\lambda_t)''$ where  $(U^\lambda_t)$ is given by Equation $(\ref{periodic})$. It is (up to state-preserving isomorphism) the only free Araki-Woods factor of type ${\rm III_\lambda}$.}
\end{nota}

\begin{nota}\label{araki}
\emph{More generally, for any nontrivial countable subgroup $\Gamma \subset \R^*_+$, we shall denote by $(T_{\Gamma}, \varphi_{\Gamma})$ the unique (up to state-preserving isomorphism) almost periodic free Araki-Woods factor whose $\Sd$ invariant is exactly $\Gamma$. Of course, $\varphi_{\Gamma}$ is its free quasi-free state. If $\Gamma = \lambda^{\Z}$ for $\lambda \in ]0, 1[$, then $(T_{\Gamma}, \varphi_{\Gamma})$ is of type ${\rm III_{\lambda}}$; in this case, it will be simply denoted by $(T_{\lambda}, \varphi_{\lambda})$, as in Notation $\ref{Tlambda}$. Theorem $6.4$ in \cite{shlya97} gives the following formula:}
$$(T_{\Gamma}, \varphi_{\Gamma}) \cong \displaystyle{\mathop{\ast} _{\gamma \in \Gamma} (T_{\gamma}, \varphi_{\gamma})}.$$
\end{nota}

\subsection{Haagerup Property for Groups and Finite von Neumann Algebras}

Remind that a countable group $G$ is said to have the \emph{Haagerup property} if there exists a sequence $(\varphi_n)$ of normalized (i.e. $\varphi_n(1) = 1$, $\forall n \in \N$) positive definite functions on $G$ such that each $\varphi_n$ vanishes at infinity and $\lim_{n \to \infty} \varphi_n(g) = 1$, $\forall g \in G$. This property was proven by Haagerup in \cite{haa} for the free groups $\mathbf{F}_n$, $2 \leq n \leq \infty$. Other examples include $\SL(2, \Z)$, and more generally $\SL(2, \mathbf{F})$ for any number field $\mathbf{F}$ (see \cite{haagerup} for a more comprehensive list of groups with the Haagerup property).

This notion can be extended to finite von Neumann algebras. Let $(N, \tau)$ be a finite von Neumann algebra. Let $\phi : N \to N$ be a completely positive map, and assume that there exists $c > 0$ such that $\tau \circ \phi \leq c \tau$. Let $T_\phi$ be the linear operator on $L^2(N, \tau)$ defined by 
\begin{equation*}
T_\phi \widehat{x} = \widehat{\phi(x)}, \forall x \in N.
\end{equation*}
We can check that $T_\phi$ is bounded and precisely $\|T_\phi\| \leq c \|\phi(1)\|$. We have the following definition:
\begin{df} [Choda, \cite{choda}]
\emph{Let $N$ be a finite von Neumann algebra. The von Neumann algebra $N$ is said to have the \emph{Haagerup property} if there exist a faithful normal trace $\tau$ on $N$ and a sequence of normal completely positive maps $\phi_n : N \to N$ such that }
\begin{enumerate}
\item $\tau \circ \phi_n \leq \tau$, $\phi_n(1) \leq 1$, $\forall n \in \N$;
\item \emph{the corresponding operator $T_{\phi_n}$ on $L^2(N, \tau)$ is compact, $\forall n \in \N$;}
\item $\| \phi_n(x) - x \|_2 \to 0$, $\forall x \in N$.
\end{enumerate}
\end{df}

It was shown by Jolissaint in \cite{jolissaint} that this property does not depend on the faithful normal trace on $N$, i.e. if $N$ has the Haagerup property, then for any faithful normal trace $\tau$ on $N$, there exists a sequence $(\phi_n)$ of completely positive maps $\phi_n : N \to N$ such that conditions $(1-3)$ are satisfied. It was proven in \cite{choda} that a group $G$ has the Haagerup property iff $L(G)$ has. If $N$ has the Haagerup property, and $p$ is a nonzero projection in $N$, then $pNp$ has the Haagerup property.

Any amenable finite von Neumann algebra has the Haagerup property. Moreover, any interpolated free group factor $L(\mathbf{F}_t)$ ($1 < t \leq \infty$) has the Haagerup property \cite{{dykema94}, {radulescu1994}}.

\subsection{Relative Property $(T)$ for Groups and Finite von Neumann Algebras} 

Let $G$ be a countable discrete group and let $H \subset G$ be a subgroup. The pair $(G, H)$ is said to have the \emph{relative property} $(T)$ if one of the following equivalent conditions is satisfied \cite{HV}:
\begin{enumerate}
\item [$\bullet$] Any unitary representation $\pi$ of $G$ which has almost invariant vectors, has a nonzero $H$-invariant vector.
\item [$\bullet$] Whenever a sequence $(\varphi_n)$ of normalized, positive definite functions on $G$ converges to $1$, then it converges to $1$ uniformly on $H$.
\end{enumerate}
A group $G$ has \emph{property} $(T)$ if the pair $(G, G)$ has the relative property $(T)$. As we mentioned in the introduction, the pair $(\Z^2 \rtimes \SL(2, \Z), \Z^2)$ has the relative property $(T)$, and more generally the pair $(\Z^2 \rtimes \Gamma, \Z^2)$ has the relative property $(T)$, for any nonamenable subgroup $\Gamma \subset \SL(2, \Z)$ (see \cite{{burger}, {shalom}}). Other examples were given by Valette \cite{valetterel} and Fern\'os \cite{fernos}. Remind that $\SL(n, \Z)$ has property $(T)$, for every $n \geq 3$. Clearly, relative property $(T)$ is an obstruction to the Haagerup property; more precisely if the group $G$ contains an infinite subgroup $H$ with the relative property $(T)$,  then $G$ cannot have the Haagerup property.

Property $(T)$ was defined by Connes and Jones in \cite{CJ} for type ${\rm II_1}$ factors. More generally, in \cite{popa2001}, Popa defined the relative property $(T)$ for inclusions of finite von Neumann algebras. It naturally involves the language of correspondences and completely positive maps.
\begin{term}
\emph{
Let $(P, \tau)$ be a finite von Neumann algebra endowed with a faithful normal trace. Let $\mathcal{H}$ be a $P$-$P$ bimodule. Let $(\xi_n)$ be a sequence of unit vectors in $\mathcal{H}$. We say that 
\begin{enumerate}
\item [$\bullet$] $(\xi_n)$ is almost central if $\|x \xi_n - \xi_n x\| \to 0$, $\forall x \in P$.
\item [$\bullet$] $(\xi_n)$ is almost $\tau$-tracial if $\| \langle \cdot \xi_n, \xi_n \rangle - \tau\| \to 0$ and $\| \langle \xi_n \cdot, \xi_n \rangle - \tau \| \to 0$.
\end{enumerate}
Let $Q \subset P$ be a von Neumann subalgebra. A unit vector $\xi \in \mathcal{H}$ is said to be $Q$-central if $x \xi = \xi x$, $\forall x \in Q$.} 
\end{term}

\begin{df} [Popa, \cite{popa2001}]
\emph{Let $P$ be a finite von Neumann algebra. Let $Q \subset P$ be a von Neumann subalgebra. Denote by $(Q)_1$ the unit ball of $Q$ (w.r.t. the operator norm). The inclusion $Q \subset P$ is said to be \emph{rigid} or to have the \emph{relative property} $(T)$ if one of the following equivalent conditions holds:}
\begin{enumerate}
\item \emph{There exists a faithful normal trace $\tau$ on $P$, such that for any $P$-$P$ bimodule $\mathcal{H}$, if $\mathcal{H}$ contains a sequence of almost central, almost $\tau$-tracial vectors, then it contains a sequence of $Q$-central, almost $\tau$-tracial vectors.}
\item \emph{There exists a faithful normal trace $\tau$ on $P$, such that for any sequence $(\phi_n)$ of normal completely positive maps $\phi_n : P \to P$, such that for any $n \in \N$, $\phi_n(1) \leq 1$, $\tau \circ \phi_n \leq \tau$, the following holds true:}
\begin{equation*}
\emph{\mbox{if }} \forall x \in P, \| \phi_n(x) - x \|_2 \to 0, \emph{\mbox{ then }} \sup_{x \in (Q)_1} \| \phi_n(x) - x \|_2 \to 0.
\end{equation*}
\item \emph{Condition $(1)$ above is satisfied for any faithful normal trace $\tau$ on $P$.}
\item \emph{Condition $(2)$ above is satisfied for any faithful normal trace $\tau$ on $P$.}
\end{enumerate}
\emph{The von Neumann algebra $P$ is said to have \emph{property} $(T)$ if the inclusion $P \subset P$ is rigid.}
\end{df}

It was proven in \cite{popa2001} that the pair $(G, H)$ has the relative property $(T)$ iff the inclusion $L(H) \subset L(G)$ is rigid. This notion of rigid inclusion behaves well w.r.t. compressions. Namely, take $q \in Q$ a nonzero projection. If the inclusion $Q \subset P$ is rigid, then the inclusion $qQq \subset qPq$ is rigid \cite{popa2001}. Finally, we remind the following theorem which will be needed in Section $\ref{fun}$:  for finite von Neumann algebras, relative property $(T)$ is an obstruction to Haagerup property. 

\begin{theo}[Popa, \cite{popa2001}]\label{haagfac}
Let $P, Q$ be finite von Neumann algebras. Assume that $Q$ is diffuse and $Q \subset P$ is a rigid inclusion. Then $P$ cannot have the Haagerup property.
\end{theo}

\section{Main Examples of (S-)Malleable Freely Mixing Actions}\label{exa}

\subsection{Bogoliubov Shifts on the Almost Periodic Free Araki-Woods Factors}
Let $G$ be any countable discrete group, let $H_\R$ be a separable real Hilbert space of infinite dimension. Let $\pi : G \to O(H_\R)$ be an orthogonal representation. Let $(U_t)$ be an almost periodic orthogonal representation of $\R$ on $H_\R$ and denote by  $(\mathcal{N}, \varphi) = (\Gamma(H_\R, U_t)'', \varphi_U)$ the associated free Araki-Woods factor. We shall always assume that $\pi$ and $(U_t)$ commute. Remind that the representation $\pi$ is said to be $C_0$, if for any $\xi, \eta \in H_\R$, $\langle \pi(g)\xi, \eta\rangle \to 0$, as $g \to \infty$. The following construction gives an example where $\pi$ and $(U_t)$ commute.

\begin{exam}
\emph{Let $\rho : G \to O(K_\R)$ be any orthogonal representation of $G$ on an infinite dimensional separable real Hilbert space $K_\R$. Let $\Gamma \subset \R^*_+$ be a countable subgroup. Let $(U^\Gamma_t)$ be the orthogonal representation on $H_\R = \bigoplus_{\gamma \in \Gamma}(K_\R \oplus K_\R)$ defined by $U_t^\Gamma = \bigoplus_{\gamma \in \Gamma} U^\gamma_t$ (see Notation $\ref{Tlambda}$). We know that $(T_\Gamma, \varphi_\Gamma) \cong (\Gamma(H_\R, U^\Gamma_t)'', \varphi_{U^\Gamma})$.  Consider now $\pi = \bigoplus_{\gamma \in \Gamma}(\rho \oplus \rho)$ on $H_\R$. It is clear that the representations $\pi$ and $(U_t^\Gamma)$ commute. Note that if the representation $\rho$ is $C_0$, then $\pi$ is $C_0$.}
\end{exam}

Denote by $H_\C$ the complexified Hilbert space of $H_\R$. The complexified representations are still denoted by $\pi$ and $(U_t)$. Denote by $A$ the infinitesimal generator of $(U_t)$ on $H_\C$: 
\begin{equation*}
U_t = A^{it}, \forall t \in \R.
\end{equation*}
By functional calculus, we know that the (unbounded) operator $A$ is affiliated with the von Neumann algebra $\pi(G)'$. Thus the (bounded) operator $\frac{2}{1 + A^{-1}}$ belongs to $\pi(G)'$. Consequently, for any $g \in G$ and any $\xi, \eta \in H_\C$, we have
\begin{eqnarray*}
\langle \pi(g)\xi, \pi(g)\eta \rangle_U & = & \langle \frac{2}{1 + A^{-1}}\pi(g)\xi, \pi(g)\eta \rangle \\
& = & \langle \pi(g)\frac{2}{1 + A^{-1}}\xi, \pi(g)\eta \rangle \\
& = & \langle \frac{2}{1 + A^{-1}}\xi, \eta \rangle \\
& = & \langle \xi, \eta \rangle_U.
\end{eqnarray*}
Thus, the representation $\pi$ is unitary w.r.t. the inner product $\langle \cdot, \cdot \rangle_U$. Note that if $\pi$ is $C_0$ w.r.t. the inner product $\langle \cdot , \cdot \rangle$, then $\pi$ is still $C_0$ w.r.t. the inner product $\langle \cdot , \cdot \rangle_U$. Denote as before, by $H$ the completion of $H_\C$ w.r.t.  $\langle \cdot, \cdot \rangle_U$. For any $g \in G$, set
\begin{equation*}
w_g = 1 \oplus \bigoplus_{n \geq 1} \pi(g)^{\otimes n}.
\end{equation*}
It is clear that for every $g \in G$, $w_g$ is a unitary of $\mathcal{F}(H)$, the full Fock space of $H$. For every $g \in G$, write $\sigma_g^\pi = \Ad(w_g)$. If no confusion is possible, $\sigma_g^\pi$ will be simply denoted by $\sigma$. It is straightforward to check that for any $\xi \in H$, $\sigma_g(l(\xi)) = l(\pi(g)\xi)$. Therefore, $(\sigma_g)$ defines an action on $(\mathcal{N}, \varphi) := (\Gamma(H_\R, U_t)'', \varphi_U)$, called the \emph{free Bogoliubov shift}. Moreover, since $w_g\Omega = \Omega$, $\forall g \in G$, the action $(\sigma_g)$ is $\varphi$-preserving.

\begin{prop}\label{mallea}
The action $(\sigma_g)$ is s-malleable for $\varphi$.
\end{prop}

\begin{proof}
From \cite{shlya97}, we know that $\Gamma(H_\R \oplus H_\R, U_t \oplus U_t)'' \cong (\mathcal{N}, \varphi) \ast (\mathcal{N}, \varphi)$. Consider on the real Hilbert space $H_\R \oplus H_\R$, the following family of orthogonal elements:
\begin{equation*}
V_t = \begin{pmatrix}
\cos(\frac{\pi}{2}t) & -\sin(\frac{\pi}{2}t) \\
\sin(\frac{\pi}{2}t) & \cos(\frac{\pi}{2}t)
\end{pmatrix}, \forall t \in \R.
\end{equation*}
It is clearly an orthogonal representation of $\R$ on $H_\R \oplus H_\R$. Consider now the canonical action $(\alpha_t)$ on $(\mathcal{N}, \varphi) \ast (\mathcal{N}, \varphi)$ associated with $(V_t)$:
\begin{equation*}
\alpha_t(s(\begin{pmatrix}
\xi \\ \eta \end{pmatrix})) =  s(V_t\begin{pmatrix}
\xi \\ \eta \end{pmatrix}), \forall t \in \R, \forall \xi, \eta \in H_\R.
\end{equation*}
We can easily see that $(V_t)$ commutes with $(U_s \oplus U_s)$ and with $\pi \oplus \pi$; consequently, the action $(\alpha_t)$ is $\varphi \ast \varphi$-preserving and commute with the diagonal action $(\sigma_g \ast \sigma_g)$. Moreover, $\alpha_1(a \ast 1) = 1 \ast a$, for every $a \in \mathcal{N}$. At last,  consider the automorphism $\beta$ defined on $(\mathcal{N}, \varphi) \ast (\mathcal{N}, \varphi)$ by:
\begin{equation*}
\beta(s(\begin{pmatrix}
\xi \\ \eta \end{pmatrix})) = s(\begin{pmatrix}
\xi \\ -\eta \end{pmatrix}), \forall \xi, \eta \in H_\R. 
\end{equation*}
It is straightforward to check that $\beta$ commutes with the diagonal action $(\sigma_g \ast \sigma_g)$, $\beta^2 = \Id$, $(\mathcal{N} \ast \C) \subset (\mathcal{N} \ast \mathcal{N})^\beta$, and $\beta\alpha_{t} = \alpha_{-t}\beta$, $\forall t \in \R$. We are done.
\end{proof}

\begin{prop}
The action $(\sigma_g)$ is freely mixing for the state $\varphi$ if and only if the representation $\pi$ is $C_0$.
\end{prop}

\begin{proof}
Assume first that $(\sigma_g)$ is freely mixing. Define $x = s(\xi), y = s(\eta)$ for $\xi, \eta \in H_\R$. We know that $\varphi(x) = \varphi(y) = 0$. Moreover, we have
\begin{eqnarray*}
\langle \pi(g)\xi, \eta\rangle & = & 4 \varphi\left(s(\pi(g)\xi) s(\eta)\right) \\
& = & 4 \varphi \left(\sigma_g(s(\xi))s(\eta)\right) \\
& = & 4 \varphi(\sigma_g(x)y).  
\end{eqnarray*}
Consequently, $\langle \pi(g)\xi, \eta\rangle \to 0$, as $g \to \infty$ and the representation $\pi$ is $C_0$.

Conversely, assume now that the representation $\pi$ is $C_0$. We have to prove that for any $n \in \N^*$, and for any $x_1, \dots, x_n$, $y_0, \dots, y_n \in~\mathcal{N} \ominus~\C$, except possibly $y_0$ and/or $y_n$ are equal to $1$,
\begin{equation*}
\lim_{g_1, \dots, g_n \to \infty} \varphi(y_0\sigma_{g_1}(x_1)y_1 \cdots \sigma_{g_n}(x_n)y_n) = 0.
\end{equation*}
It suffices to show this property for $x_i$ and $y_j$ words in $s(h_\iota)$. For $1 \leq i \leq n$ and $0 \leq j \leq n$, take
\begin{eqnarray*}
x_i & = & s(\xi^i_1) \cdots s(\xi^i_{r_i}),\\
y_j & = & s(\eta^j_1) \cdots s(\eta^j_{s_j}).
\end{eqnarray*}
Then for any $1 \leq i \leq n$ and any $g_1, \dots, g_n \in G$, 
\begin{equation*}
\sigma_{g_i}(x_i) = s(\pi(g_i)\xi^i_1) \cdots s(\pi(g_i)\xi^i_{r_i}).
\end{equation*}
For $1 \leq i \leq n$ and $0 \leq j \leq n$, define
\begin{eqnarray*}
K_{2i - 1} & = & \{1, \dots, r_i\} \\
K_{2j} & = & \{1, \dots, s_j\}.
\end{eqnarray*}
Let $m = \sum_i r_i + \sum_j s_j$. We will use the following identification
\begin{equation*}
\{1, \dots, m\} = K_0 \sqcup K_1 \sqcup K_2 \sqcup \cdots \sqcup K_{2n - 1} \sqcup K_{2n}.
\end{equation*}

If $m$ is odd, there is nothing to prove. If $m$ is even, then from Equation $(\ref{formula})$, we know that
\begin{equation}\label{formula2}
\varphi(y_0\sigma_{g_1}(x_1)y_1 \cdots \sigma_{g_n}(x_n)y_n) = 2^{-m}\sum_
{(\{\beta_i, \gamma_i\}) 
 \in NC(m),  
\beta_i < \gamma_i}
\prod_{k = 1}^{m/2}\langle h_{\beta_k}, h_{\gamma_k}\rangle_U,
\end{equation}
where the letter $h$ stands for $\eta$ or $\pi(g)\xi$. Let $\nu = (\{\beta_i, \gamma_i\}) \in NC(m)$. Write $z_{2i - 1} = x_i, z_{2j}  = y_j$, for $i \in \{1, \dots, n\}, j \in \{0, \dots, n\}$.
\begin{itemize}

\item [$(a)$] Either there exist for the noncrossing pairing $\nu$, some $i \in \{1, \dots, n\}$, $j \in \{0, \dots, n\}$, $p \in K_{2i - 1}$, and $q \in K_{2j}$ such that $p, q$ are in the same class. Thus, the inner product $\langle \pi(g_i)\xi^i_p, \eta^j_q \rangle_U$ or $\overline{\langle \pi(g_i)\xi^i_p, \eta^j_q \rangle_U}$ necessarily appears in the product $\prod_{k = 1}^{m/2}\langle h_{\beta_k}, h_{\gamma_k}\rangle_U$ of Equation $(\ref{formula2})$. Since the representation $\pi$ is $C_0$, $\langle \pi(g_i)\xi^i_p, \eta^j_q \rangle_U \to 0$, as $g_i \to \infty$.

\item [$(b)$] Or for any $p, q \in \{1, \dots, m\}$, if $p, q$ are in the same class, then it means that  necessarily $p \in K_{2i - 1}$, $q \in K_{2i' - 1}$ or $p \in K_{2j}$, $q \in K_{2j'}$ for some $i, i', j, j'$. Then there must exist a maximal $r \geq 1$ and integers $0 \leq k_1 < \cdots < k_r \leq 2n$, such that when one restricts the noncrossing pairing $\nu$ to the subsets $K_{k_1}, \dots, K_{k_r}$, one still has a noncrossing pairing on each of those subsets. Now if we sum up over all the noncrossing pairings $\nu'$ of $\{1, \dots, m\}$ such that $\nu'$ and $\nu$ agree on $\{1, \dots, m\} \backslash (K_{k_1} \sqcup \cdots \sqcup K_{k_r})$, and $\nu'$ can be any noncrossing pairing on each of the subsets $K_{k_1}, \dots, K_{k_r}$, we obtain
\begin{equation}\label{parsum}
2^{-m} \sum_{\nu'} \prod_{k = 1}^{m/2}\langle h_{\beta_k}, h_{\gamma_k}\rangle_U = C \varphi(z_{k_1}) \cdots \varphi(z_{k_r}),
\end{equation}
with the constant $C$ given by
\begin{equation*}
C = 2^{-m'} \prod_{\{(\beta_k, \gamma_k)\} \in \nu^c} \langle h_{\beta_k}, h_{\gamma_k}\rangle_U,
\end{equation*}
where $\nu^c$ denotes the restriction of $\nu$ on $\{1, \dots, m\} \backslash (K_{k_1} \sqcup \cdots \sqcup K_{k_r})$ and $m'$ is the cardinality of $\{1, \dots, m\} \backslash (K_{k_1} \sqcup \cdots \sqcup K_{k_r})$. By choice of $x_i, y_j$, we have $\varphi(z_{k_1}) = \cdots =  \varphi(z_{k_r}) = 0$, so that the sum in Equation $(\ref{parsum})$ is $0$. 
\end{itemize}
Finally if we sum up over all the noncrossing pairings, according to $(a)$ and $(b)$, we are done.
\end{proof}

For further applications (see Section $\ref{fun}$), we shall always take $\pi = \lambda_G$ the left regular representation of $G$ which is $C_0$ as soon as the group $G$ is infinite.

\subsection{Free Bernoulli Shifts}

Let $G$ be any infinite countable discrete group and let $(\mathcal{N}, \varphi)$ be any von Neumann algebra. Write 
\begin{equation*}
(\mathcal{M}, \Phi) = \mathop{\ast}_{g \in G} (\mathcal{N}, \varphi)_g.
\end{equation*}
For any $g, h \in G$ and any $x_h \in \mathcal{N}_h$, let $\sigma_g(x_h) = x_{g^{-1}h} \in \mathcal{N}_{g^{-1}h}$. The action $\sigma$ extends to the whole von Neumann algebra $\mathcal{M}$ and is called the {\it free Bernoulli shift} with base $(\mathcal{N}, \varphi)$. The action $\sigma$ is obviously $\Phi$-preserving.
\begin{prop}
The action $(\sigma_g)$ is freely mixing for $\Phi$.
\end{prop}

\begin{proof}
For $h \in G$, denote as usual $\mathcal{N}_h \ominus \C = \mathcal{N}_h \cap \ker (\Phi)$. We have to prove that for any $n \in \N^*$, $x_1, \dots, x_n, y_0, \dots, y_n \in \mathcal{M} \ominus \C$, except possibly $y_0$ and/or $y_n$ are equal to~$1$,
\begin{equation*}
\lim_{g_1, \dots, g_n \to \infty} \Phi(y_0\sigma_{g_1}(x_1)y_1 \cdots \sigma_{g_n}(x_n)y_n) = 0.
\end{equation*}
Actually, it suffices to show this property for $x_i$ and $y_j$ of the form
\begin{eqnarray*}
x_i & = & x_{h^i_1} \cdots x_{h^i_{m_i}},\\
y_j & = & y_{k^j_1} \cdots y_{k^j_{n_j}}, 
\end{eqnarray*}
with $1 \leq i \leq n$, $0 \leq j \leq n$ and
\begin{enumerate}
\item $h^i_1 \neq \cdots \neq h^i_{m_i}$ and for $1 \leq p \leq m_i$, $x_{h^i_p} \in \mathcal{N}_{h^i_p} \ominus \C$;
\item $k^j_1 \neq \cdots \neq k^j_{n_j}$ and for $1 \leq q \leq n_j$, $y_{k^j_q} \in \mathcal{N}_{k^j_q} \ominus \C$.
\end{enumerate}
For any $g_1, \dots, g_n \in G$, we have
\begin{equation*}
\sigma_{g_i}(x_i) = x_{(g_i)^{-1}h^i_1} \cdots x_{(g_i)^{-1}h^i_{m_i}}.
\end{equation*}
With all these notations, it is clear now that for any $g_1, \dots, g_n \in G$ large enough, the $(g_i)^{-1}h^i_{p}$'s are pairwise distinct from the $k^j_q$'s. In particular, using the freeness property of the state $\Phi$, for any $g_1, \dots, g_n \in G$ large enough, we get
\begin{equation*}
\Phi(y_0\sigma_{g_1}(x_1)y_1 \cdots \sigma_{g_n}(x_n)y_n) = 0.
\end{equation*}
\end{proof}

\begin{prop}
Take $(\mathcal{N}, \varphi) = (\Gamma(H_\R, U_t)'', \varphi_U)$ an almost periodic free Araki-Woods factor. Then, the associated free Bernoulli shift is s-malleable.
\end{prop}

\begin{proof}
To check that the free Bernoulli action is s-malleable, it suffices to produce an action $(\alpha_t)$ of $\R$ on $(\mathcal{N}, \varphi) \ast (\mathcal{N}, \varphi)$ and a period $2$ automorphism $\beta$ of $(\mathcal{N}, \varphi) \ast (\mathcal{N}, \varphi)$ such that all the conditions of Definition \ref{malleable} are satisfied. One can then take the infinite free product of these $(\alpha_t)$ and $\beta$. But, we have already proven this property in Proposition~$\ref{mallea}$.
\end{proof}

\begin{prop}
Take $(\mathcal{N}, \varphi) = (M_n(\C), \omega)$, where $\omega$ is any faithful state on $M_n(\C)$. Then, the  associated free Connes-St\o rmer Bernoulli shift is malleable.
\end{prop}

\begin{proof}
Exactly in the same way, to check that the free Bernoulli action is malleable, it suffices to produce an action $(\alpha_t)$ of $\R$ on $(M_n(\C), \omega) \ast (M_n(\C), \omega)$ such that $\alpha_1(a \ast 1) = 1 \ast a$, $\forall a \in M_n(\C)$. Dykema proved in \cite{dykema96}, among other results, that the centralizer $(M_n(\C) \ast M_n(\C))^{\omega \ast \omega}$ is the type ${\rm II_1}$ factor $L(\F_\infty)$. We should mention that we obtained in \cite{houdayer2} classification results for some of these free products using free Araki-Woods factors of Shlyakhtenko. 

Denote by $(e_{ij})$ for $i, j \in \{0, \dots, n - 1\}$ a system of matrix unit in $M_n(\C)$. Prove the following lemma:
\begin{lem}\label{plomod}
Let $(\mathcal{P}, \psi)$ be a von Neumann algebra endowed with a faithful normal state, such that the centralizer $\mathcal{P}^\psi$ is a factor. For $i = 1, 2$, let $\rho_i : M_n(\C) \hookrightarrow (\mathcal{P}, \psi)$ be a \emph{modular} embedding, i.e. $\rho_i$ is state-preserving and  $\rho_i(M_n(\C))$ is globally invariant under the modular group $(\sigma_t^\psi)$. Then, there exists a unitary $u \in \mathcal{P}^\psi$ such that $u \rho_1(e_{ij}) u^* = \rho_2(e_{ij})$, $\forall i, j \in \{0, \dots, n - 1\}$.
\end{lem}
\begin{proof}[Proof of Lemma $\ref{plomod}$]
Let $i \in \{1, 2\}$. Denote by $p_i = \rho_i(e_{00})$. Since $\rho_i$ is modular, we have $p_{1, 2} \in \mathcal{P}^\psi$ and $\psi(p_1) = \psi(p_2) = \omega(e_{00})$. Since $\mathcal{P}^\psi$ is a factor, there exists a partial isometry $v \in \mathcal{P}^\psi$ such that $p_1 = v^*v$ and $p_2 = vv^*$. Denote by $u = \sum_{i = 0}^{n - 1} \rho_2(e_{i0}) v \rho_1(e_{0i})$. An easy computation shows that $u$ is a unitary and $u \in \mathcal{P}^\psi$, since $\rho_{1, 2}$ are modular. Moreover, for any $0 \leq k, l \leq n - 1$, $u \rho_1(e_{kl}) u^* = \rho_2(e_{kl})$. 
\end{proof}

Write $(\mathcal{P}, \psi) = (M_n(\C), \omega) \ast (M_n(\C), \omega)$. Define $\rho_{1, 2} : M_n(\C) \to \mathcal{P}$, by $\rho_1(x) = x \ast 1$, $\rho_2(x) = 1 \ast x$, $\forall x \in M_n(\C)$. The embeddings $\rho_{1, 2}$ are modular (see \cite{dykema96}). Applying Lemma $\ref{plomod}$, there exists $u \in \mathcal{U}(\mathcal{P}^\psi)$ such that $u \rho_1(e_{ij}) u^* = \rho_2(e_{ij})$, for any $0 \leq i, j \leq n - 1$. Write $u = \exp(ih)$ with $h$ a selfadjoint element in $\mathcal{P}^\omega$. Denote by $u_t = \exp(ith)$, $\forall t \in \R$. We can then define on $\mathcal{P}$, $\alpha_t = \Ad(u_t)$. Since $u_t \in \mathcal{P}^\omega$, the action $(\alpha_t)$ is $\omega$-preserving. By definition, $\alpha_1(x \ast 1) = 1 \ast x$, $\forall x \in M_n(\C)$. We are done.
\end{proof}

\section{Popa's Intertwining Techniques}\label{te}

\subsection{Intertwining Techniques for von Neumann Algebras Endowed with Almost Periodic States.}

We remind Popa's intertwining-by-bimodules technique: it is a very strong method to prove that two von Neumann subalgebras of a von Neumann algebra are unitarily conjugate. Roughly, Definition $\ref{embed}$ below says the following. Let $A, B \subset M$ be von Neumann subalgebras of a finite von Neumann algebra $(M, \tau)$. The following conditions are equivalent:
\begin{enumerate}
\item [$\bullet$] A corner of $A$ can be conjugated into a corner of $B$.
\item [$\bullet$] The $A$-$B$ bimodule $L^2(M, \tau)$ contains a nonzero $A$-$B$ subbimodule which is finitely generated as a right $B$-module.
\item [$\bullet$] The basic construction $\langle M, e_B \rangle$ contains a positive element $a$ commuting with $A$ and satisfying $0 < \widehat{\tau}(a) < \infty$, where $\widehat{\tau}$ denotes the canonical semifinite trace on the basic construction $\langle M, e_B \rangle$. 
\end{enumerate}

\begin{term}
\emph{Let $\mathcal{M}$ be a von Neumann algebra. For a possibly non-unital subalgebra $Q \subset \mathcal{M}$, we shall denote by $1_Q$ the unit of $Q$. Obviously, $1_Q$ is a projection in $\mathcal{M}$ and $Q \subset 1_Q \mathcal{M} 1_Q$. We shall always mention when a von Neumann subalgebra is possibly non-unital.
}
\end{term}

\begin{df}[Popa, \cite{{popamal1}, {popa2001}}]\label{embed}
\emph{Let $(\mathcal{M}, \varphi)$ be a von Neumann algebra endowed with an almost periodic state. Denote by $\| \cdot \|_2$ the $L^2$-norm w.r.t. the state $\varphi$. Assume that}
\begin{enumerate}
\item [$\bullet$] \emph{$A \subset \mathcal{M}^\varphi$ is a possibly non-unital von Neumann subalgebra, and denote by $1_A$ its unit;}
\item [$\bullet$] \emph{$B \subset \mathcal{M}^\varphi$ is a unital von Neumann subalgebra.}
\end{enumerate}
\emph{We say that $A$ \emph{embeds into} $B$ \emph{inside} $\mathcal{M}$ and write $\displaystyle{A \mathop{\prec}_{\mathcal{M}} B}$, if one of the following equivalent conditions is satisfied:}
\begin{enumerate}
\item  \emph{There exist $n \geq 1$, $\gamma > 0$, $v \in M_{1, n}(\C) \otimes 1_A \mathcal{M}$, a projection $p \in B^n$ and a (unital) $\ast$-homomorphism $\theta : A \to pB^np$ such that $v$ is a nonzero partial isometry which is a $\gamma$-eigenvector for $\varphi$, $v^*v \leq p$ and}
\begin{equation*}
xv = v\theta(x), \forall x \in A.
\end{equation*}
\item  \emph{There exists a nonzero element $w \in 1_A \mathcal{M}$ such that $Aw \subset \sum_{k = 1}^{n}w_k B$ for finitely many $w_k \in 1_A \mathcal{M}$.}
\item \emph{There exists a nonzero element $a \in 1_A \langle \mathcal{M}, e_B\rangle^+1_A \cap A'$ with $\widehat{\varphi}_B(a) < \infty$. Here $\langle\mathcal{M}, e_B\rangle$ denotes the basic construction for the inclusion $B \subset \mathcal{M}$, with its canonical almost periodic semifinite weight $\widehat{\varphi}_B$.}
\item \emph{There is no sequence of unitaries $(u_k)$ in $P$ such that $\|E_B(a^* u_k b)\|_2 \to 0$ for all $a, b \in 1_A \mathcal{M}$.}
\end{enumerate}
\end{df}

We refer to Theorem $2.1$ in \cite{popamal1} for the proof of these properties (see also Proposition C.$1$ in \cite{vaesbern}). Note that if $B = \C$, then $\displaystyle{A \mathop{\nprec}_{\mathcal{M}} \C}$ if and only if $A$ is diffuse. Indeed, if $B = \C$, we simply have $E_B = \varphi 1$. Since the von Neumann algebra $A$ is finite, $A$ is diffuse if and only if there exists a sequence of unitaries $(u_k)$ in $A$ that weakly tends to $0$. But,
\begin{eqnarray*}
u_k \to 0 \mbox{ weakly } & \Longleftrightarrow & \langle u_k\xi, \eta\rangle \to 0, \forall \xi, \eta \in 1_A L^2(\mathcal{M}, \varphi) \\
& \Longleftrightarrow & \langle u_k \widehat{b}, \widehat{a} \rangle \to 0, \forall a, b \in 1_A \mathcal{M} \mbox{ (since } (u_k) \mbox{ is bounded) } \\
& \Longleftrightarrow & \varphi(a^* u_k b) \to 0, \forall a, b \in 1_A \mathcal{M}.
\end{eqnarray*}

For our purpose, we need a generalization of this technique. Indeed, we want to allow the subalgebra $B$ to be globally invariant under the modular group $(\sigma^\varphi_t)$, and not just included in the centralizer $\mathcal{M}^\varphi$. We prove the following theorem:

\begin{theo}\label{newembed}
Let $(\mathcal{M}, \varphi)$ be a von Neumann algebra endowed with an almost periodic state. Denote by $\| \cdot \|_2$ the $L^2$-norm w.r.t. the state $\varphi$. Assume that 
\begin{enumerate}
\item [$\bullet$] $P \subset \mathcal{M}^\varphi$ is a possibly non-unital von Neumann subalgebra, and denote by $1_P$ its unit;
\item [$\bullet$] $\mathcal{B} \subset \mathcal{M}$ is a unital von Neumann subalgebra globally invariant under the modular group $(\sigma_t^\varphi)$.
\end{enumerate}
Denote by $B = \mathcal{B}^\varphi = \mathcal{B} \cap \mathcal{M}^\varphi$. The following two conditions are equivalent:
\begin{enumerate}

\item There exist $n \geq 1$, $\gamma > 0$, $v \in M_{1, n}(\C) \otimes 1_P \mathcal{M}$, a projection $p \in B^n$ and a (unital) $\ast$-homomorphism $\theta : P \to pB^np$ such that $v$ is a nonzero partial isometry which is a $\gamma$-eigenvector for $\varphi$, $v^*v \leq p$ and 
\begin{equation*}
xv = v\theta(x), \forall x \in P.
\end{equation*}
\item There is no sequence of unitaries $(u_k)$ in $P$ such that $\|E_{\mathcal{B}}(a^* u_k b)\|_2 \to 0$  for all $a, b \in 1_P \mathcal{M}$.
\end{enumerate}
If one of the conditions holds, we shall still write $\displaystyle{P \mathop{\prec}_{\mathcal{M}} \mathcal{B}}$.
\end{theo}

\begin{proof} 
We must pay attention to the following fact: there are two different basic constructions here. The one with $\mathcal{B}$ and the other one with $B = \mathcal{B}^\varphi$. Of course, we have the inclusion $\langle \mathcal{M}, e_{\mathcal{B}} \rangle \subset \langle \mathcal{M}, e_B \rangle$, but the associated weights are not equal on $ \langle \mathcal{M}, e_{\mathcal{B}} \rangle$. For this reason, we shall denote by $\widehat{\varphi}_{\mathcal{B}}$ the weight for the basic construction $\langle \mathcal{M}, e_{\mathcal{B}} \rangle$ and by $\widehat{\varphi}_B$ the weight for $ \langle \mathcal{M}, e_B \rangle$.

$(1) \Longrightarrow (2)$. Suppose that we have all the data of $(1)$. Let $(u_k)$ be a sequence of unitaries in $P$ such that $\|E_{\mathcal{B}}(a^* u_k b)\|_2 \to 0$ for all $a, b \in 1_P \mathcal{M}$. Then  $\|(\id \otimes E_{\mathcal{B}})(v^* u_k v)\|_{\tr_n \otimes \varphi} \to 0$. But for every $k \in \N$, $v^* u_k v = \theta(u_k)v^*v$. Moreover, $\theta(u_k) \in \mathcal{U}(p B^n p)$ and $v^*v \leq p$. Thus,
\begin{eqnarray*}
\|(\id \otimes E_{\mathcal{B}})(v^* v)\|_{\tr_n \otimes \varphi} & = & \|\theta(u_k)(\id \otimes E_{\mathcal{B}})(v^* v)\|_{\tr_n \otimes \varphi}\\
& = & \|(\id \otimes E_{\mathcal{B}})(\theta(u_k) v^* v)\|_{\tr_n \otimes \varphi}\\
& = & \|(\id \otimes E_{\mathcal{B}})(v^* u_k v)\|_{\tr_n \otimes \varphi} \to 0.
\end{eqnarray*}
We conclude that $(\id \otimes E_{\mathcal{B}})(v^* v) = 0$ and so $v = 0$. Contradiction.

$(2) \Longrightarrow (1)$. We prove this implication in three steps. For any $\gamma \in \Sp(\mathcal{M}, \varphi)$, denote by $\mathcal{M}^\gamma$ the vector space of all $\gamma$-eigenvectors for $\varphi$ in $\mathcal{M}$. Since $\varphi$ is almost periodic,
\begin{equation*}
L^2(\mathcal{M}, \varphi) = \bigoplus_{\gamma \in \Sp(\mathcal{M}, \varphi)} L^2(\mathcal{M}^\gamma).
\end{equation*}
Denote by $p_\gamma$ the orthogonal projection from $L^2(\mathcal{M})$ onto $L^2(\mathcal{M}^\gamma)$.

{\bf Step (1): Proving that for any $\gamma, \lambda \in \Sp(\mathcal{M}, \varphi)$, and for any $a \in \mathcal{M}^\lambda$, we have}
\begin{equation*}
\widehat{\varphi}_B(p_\gamma a e_{\mathcal{B}} a^* p_\gamma) \leq \varphi(aa^*).
\end{equation*}

First of all, note that for any $\gamma \in \Sp(\mathcal{M}, \varphi)$, $\mathcal{B}^\gamma = \mathcal{M}^\gamma \cap \mathcal{B}$. Since $\mathcal{B}$ is globally invariant under $\sigma^\varphi$, for every $t \in \R$, we have $\sigma_t^\varphi \circ E_{\mathcal{B}} = E_{\mathcal{B}} \circ \sigma_t^\varphi$. It follows immediatly that $E_{\mathcal{B}}(\mathcal{M}^\gamma) = \mathcal{B}^\gamma$. It is straightforward to check that $e_{\mathcal{B}} p_\gamma = p_\gamma e_{\mathcal{B}}$; we shall denote this projection by $e_\mathcal{B}^\gamma$. In fact, it is nothing but the orthogonal projection of $L^2(\mathcal{M})$ onto $L^2(\mathcal{B}^\gamma)$. Take now $\gamma \in \Sp(\mathcal{B}, \varphi)$. Since $\mathcal{B}^\gamma \neq 0$, take $(v_i)_{i \in I}$ a maximal family of nonzero partial isometries in $\mathcal{B}^\gamma$ such that the final projections $p_i = v_i v^*_i$ are pairwise orthogonal. We assume that $I = \{1, \dots, n\}$ with $1 \leq n \leq \infty$. Denote by $v = [v_1 \cdots v_n] \in M_{1, n}(\C) \otimes \mathcal{B}^\gamma$ and by $p = \sum_i v_iv_i^*$. It is easy to see that 
\begin{equation*}
e_{\mathcal{B}}^\gamma \eta = v (1 \otimes e_B) v^* \eta = 0, \forall \eta \in (L^2(\mathcal{M}) \ominus L^2(\mathcal{B})) \oplus \bigoplus_{\lambda \neq \gamma} L^2(\mathcal{B^\lambda}).
\end{equation*}
Assume now that there exists $x \in \mathcal{B}^\gamma$ such that $\widehat{x} \neq v(1 \otimes e_B)v^*\widehat{x}$. Thus, $(1 - p) x \neq 0$. Write $(1 - p) x = w b$ its polar decomposition. Since $(1 - p) x \in \mathcal{B}^\gamma$, it follows that $w \in \mathcal{B}^\gamma$. But $w \neq 0$, and $ww^* \leq 1 - p$. Since the family $(v_i)_{i \in I}$ is assumed to be maximal, we have a contradiction. Consequently, we have just proven that
\begin{equation*}
e_{\mathcal{B}}^\gamma = v (1 \otimes e_B) v^*.
\end{equation*}

Since $B \subset \mathcal{M}^\varphi$, it follows that $p_\gamma \in \langle \mathcal{M}, e_B \rangle$, for all $\gamma \in \Sp(\mathcal{M}, \varphi)$. Take $\gamma, \lambda \in \Sp(\mathcal{M}, \varphi)$. Let $a \in \mathcal{M}^\lambda$. We want to prove now that $\widehat{\varphi}_B(p_\gamma a e_{\mathcal{B}} a^* p_\gamma) \leq \varphi(aa^*)$. It is easy to see that $a^* p_\gamma = p_{\gamma \lambda^{-1}} a^*$. Consequently, we have
\begin{equation*}
p_\gamma a e_{\mathcal{B}} a^* p_\gamma = a e_{\mathcal{B}}^{\gamma \lambda^{-1}} a^*.
\end{equation*}
If $\gamma \lambda^{-1} \notin \Sp(\mathcal{B}, \varphi)$, then $e_{\mathcal{B}}^{\gamma \lambda^{-1}} = 0$ and so $\widehat{\varphi}_B(p_\gamma a e_{\mathcal{B}} a^* p_\gamma) = 0$. If $\gamma \lambda^{-1} \in \Sp(\mathcal{B}, \varphi)$, take as before $v = [ v_1 \cdots v_n ] \in M_{1, n}(\C) \otimes \mathcal{B}^{\gamma \lambda^{-1}}$ such that $e_{\mathcal{B}}^{\gamma \lambda^{-1}} = v (1 \otimes e_B) v^*$ and $(v_i)$ is a family of nonzero partial isometries such that the projections $v_iv_i^*$ are pairwise orthogonal. Thus,

\begin{eqnarray*}
\widehat{\varphi}_B(p_\gamma a e_{\mathcal{B}} a^* p_\gamma) & = & \widehat{\varphi}_B(a e_{\mathcal{B}}^{\gamma \lambda^{-1}} a^*)\\
& = & \widehat{\varphi}_B(a v (1 \otimes e_B) v^* a^*) \\
& = & \widehat{\varphi}_B (a \left ( \sum_i v_i e_B v_i^* \right) a^*) \\
& = & \sum_i \widehat{\varphi}_B (a v_i e_B v_i^* a^*) \,(\widehat{\varphi}_B \mbox{ is normal})\\
& = & \sum_i \varphi(a v_i v_i^* a^*) \\
& = & \varphi (a \left ( \sum_i v_i v_i^* \right) a^*) \,(\widehat{\varphi}_B \mbox{ is normal}) \\
& \leq & \varphi(aa^*).
\end{eqnarray*}

{\bf Step (2): Finding a nonzero element $d \in 1_P \langle \mathcal{M}, e_{\mathcal{B}} \rangle^+1_P  \cap P'$  such that for any $\gamma \in \Sp(\mathcal{M}, \varphi)$, we have}
\begin{equation*}
\widehat{\varphi}_B(p_\gamma d p_\gamma) <~\infty.
\end{equation*}

By $(2)$, we can take $\varepsilon > 0$ and $K \subset 1_P \mathcal{M}$ finite subset such that for all unitaries $u \in P$, $\max_{a, b \in K} \|E_{\mathcal{B}}(a^* u b)\|_2 \geq \varepsilon$. Note that
\begin{eqnarray*}
\|E_{\mathcal{B}}(a^* u b)\|_2^2 & = & \varphi(E_{\mathcal{B}} (a^* u b)^* E_{\mathcal{B}}(a^* u b))\\
& = & \widehat{\varphi}_{\mathcal{B}}(E_{\mathcal{B}} (a^* u b)^* e_\mathcal{B}  E_{\mathcal{B}}(a^* u b))\\
& = & \widehat{\varphi}_{\mathcal{B}}(e_{\mathcal{B}} (a^* u b)^* e_\mathcal{B} (a^* u b) e_{\mathcal{B}}).
\end{eqnarray*}
Since the functional $\widehat{\varphi}_{\mathcal{B}}(e_{\mathcal{B}} \cdot e_{\mathcal{B}})$ is a normal   state on the basic construction $\langle \mathcal{M}, e_{\mathcal{B}}\rangle$ and since $\varphi$ is almost periodic, we can assume that all the elements of $K$ are eigenvectors for $\varphi$. Define now the element $c = \sum_{a \in K} a e_{\mathcal{B}}a^*$ in $1_P\langle \mathcal{M}, e_{\mathcal{B}}\rangle^+1_P$. Note that $\widehat{\varphi}_{\mathcal{B}}(c) = \sum_{a \in K} \varphi(aa^*)$, and so $\widehat{\varphi}_{\mathcal{B}}(c) < \infty$. Moreover, since the elements of $K$ are eigenvectors for $\varphi$, we get $c \in \langle \mathcal{M}, e_{\mathcal{B}}\rangle^{\widehat{\varphi}_{\mathcal{B}}}$. 
Denote by $\mathcal{C}$ the convex hull of $\{u^*cu : u \in \mathcal{U}(P)\}$. Denote now by $\overline{\mathcal{C}}$ the closure of $\mathcal{C}$ for the weak topology. It should be noted that since $\mathcal{C}$ is bounded, $\overline{\mathcal{C}}$ is also closed for the $\sigma$-weak topology. Let $d \in 1_P\langle \mathcal{M}, e_{\mathcal{B}}\rangle^+1_P$ be the element of minimal $L^2$-norm $\| \cdot \|_{2, \widehat{\varphi}_{\mathcal{B}}}$ (w.r.t. the weight $\widehat{\varphi}_{\mathcal{B}}$) in $\overline{\mathcal{C}}$. By uniqueness of the element of minimal $L^2$-norm, it follows that $u^* d u = d$, $\forall u \in \mathcal{U}(P)$, and so $d \in 1_P\langle \mathcal{M}, e_{\mathcal{B}}\rangle^+1_P \cap P'$. Obviously, we have $\sigma_t^{\widehat{\varphi}_{\mathcal{B}}}(d) = d$, $\forall t \in \R$. We show now that $d \neq 0$. For all $u \in \mathcal{U}(P)$, we have 
\begin{eqnarray*}
\sum_{b \in K} \widehat{\varphi}_{\mathcal{B}}(e_{\mathcal{B}} b^* (u^*cu) b e_{\mathcal{B}}) & = & \sum_{a, b \in K} \widehat{\varphi}_{\mathcal{B}}(e_{\mathcal{B}} (a^* u b)^* e_\mathcal{B} (a^* u b) e_{\mathcal{B}})\\
& = & \sum_{a, b \in K} \widehat{\varphi}_{\mathcal{B}}(E_{\mathcal{B}} (a^* u b)^* e_\mathcal{B}  E_{\mathcal{B}}(a^* u b))\\
& = & \sum_{a, b \in K} \varphi(E_{\mathcal{B}} (a^* u b)^* E_{\mathcal{B}}(a^* u b))\\
& = & \sum_{a, b \in K} \|E_{\mathcal{B}}(a^* u b)\|_2^2 \geq \varepsilon^2.
\end{eqnarray*}
Consequently, we have
\begin{equation*}
\sum_{b \in K} \widehat{\varphi}_{\mathcal{B}}(e_{\mathcal{B}} b^* y b e_{\mathcal{B}}) \geq \varepsilon^2, \forall y \in \mathcal{C}.
\end{equation*}
Since the functional $\widehat{\varphi}_{\mathcal{B}}(e_{\mathcal{B}} \cdot e_{\mathcal{B}})$ is a normal   state on the basic construction $\langle \mathcal{M}, e_{\mathcal{B}}\rangle$, we get
\begin{equation*}
\sum_{b \in K} \widehat{\varphi}_{\mathcal{B}}(e_{\mathcal{B}} b^* d b e_{\mathcal{B}}) \geq \varepsilon^2.
\end{equation*} 
It follows that $d \neq 0$. At last, using the result of Step $(1)$ and since $P \subset \mathcal{M}^\varphi$, $\forall \gamma, \lambda \in \Sp(\mathcal{M}, \varphi)$, $\forall u \in \mathcal{U}(P)$, $\forall a \in 1_P \mathcal{M}^\lambda$, we have
\begin{equation*}
\widehat{\varphi}_B(p_\gamma u^* a e_{\mathcal{B}} a^* u p_\gamma) \leq \varphi(u^* a a^* u) = \varphi(a a^*).
\end{equation*}
Consequently, summing over $a \in K$ and using the convexity of $\mathcal{C}$, we obtain
\begin{equation*}
\widehat{\varphi}_B(p_\gamma y p_\gamma) \leq \sum_{a \in K} \varphi(a a^*), \forall \gamma \in \Sp(\mathcal{M}, \varphi), \forall y \in \mathcal{C}.
\end{equation*}
Using the $\sigma$-weak lower semi continuity of the weight $\widehat{\varphi}_B$ (see for example Theorem ${\rm VII}.1.11$ in \cite{takesakiII}), for every $\gamma \in \Sp(\mathcal{M}, \varphi)$, we have 
\begin{equation*}
\widehat{\varphi}_B(p_\gamma d p_\gamma) \leq \sum_{a \in K} \varphi(a a^*) < \infty.
\end{equation*}

{\bf Step (3): Constructing a nonzero $P$-$B$-subbimodule $\mathcal{H} \subset 1_P L^2(\mathcal{M}^\gamma)$ finitely generated over $B$ to conclude.} 
We remind that for any $x \in \langle \mathcal{M}, e_B \rangle^+$,
\begin{equation*}
\widehat{\varphi}_B(x) = \sum_{\gamma \in \Sp(\mathcal{M}, \varphi)} \widehat{\varphi}_B(p_\gamma x p_\gamma).
\end{equation*}
Since $d \neq 0$ and thanks to Step $(2)$, there exists $\gamma \in \Sp(\mathcal{M}, \varphi)$ such that $$0 < \widehat{\varphi}_B(p_\gamma d p_\gamma) < \infty.$$
Since $p_\gamma \in (\mathcal{M}^\varphi)'$, we have $p_\gamma \in P'$. Thus $p_\gamma d p_\gamma \in 1_P \langle \mathcal{M}, e_B \rangle^+1_P \cap P'$. Take now $q$ a nonzero spectral projection of the element $p_\gamma d p_\gamma$. We get that $\mathcal{K} = qL^2(\mathcal{M})$ is a nonzero $P$-$B$-subbimodule of $1_P L^2(\mathcal{M}^\gamma)$ with finite trace over $B$ (see the discussion in Section $\ref{ara}$). Thus, cutting down by a central projection of $B$ (see Lemma $\ref{finiment}$), we get a nonzero $P$-$B$-subbimodule $\mathcal{H} \subset 1_P L^2(\mathcal{M}^\gamma)$ which is finitely generated over $B$. Now, the rest of the proof is exactly the same as the proof of Theorem $2.1$ in \cite{popamal1}. For the sake of completeness, we proceed in order to obtain condition $(1)$.  Hence, we can take $n \geq 1$, a projection $p \in B^n$ and a right $B$-module isomorphism
\begin{equation*}
\psi : p L^2(B)^{\oplus n} \to \mathcal{H}. 
\end{equation*}
Since $\mathcal{H}$ is a $P$-module, we get a (unital) $\ast$-homomorphism $\theta : P \to p B^n p$ satisfying $x\psi(\eta) = \psi(\theta(x)\eta)$ for all $x \in P$, and $\eta \in p L^2(B)^{\oplus n}$. Define now $e_j \in L^2(B)^{\oplus n}$ as $e_j = (0, \dots, \widehat{1}, \dots, 0)$ and $\xi = (\xi_1, \dots, \xi_n) \in M_{1, n}(\C) \otimes \mathcal{H}$, with $\xi_j = \psi(pe_j)$. Let $j \in \{1, \dots, n\}$. For any $x \in P$, write $\theta(x) = (\theta_{kl}(x))_{kl} \in p B^n p$. We have
\begin{eqnarray*}
x \xi_j & = & x \psi(p e_j)\\
& = & \psi(\theta(x) p e_j)\\
& = & \psi(p\theta(x) e_j)\\
& = & \psi(p \sum_{i = 1}^n \theta_{ij}(x) e_i)\\
& = & \sum_{i = 1}^n \psi(p (0, \dots, \theta_{ij}(x), \dots, 0))\\
& = & \sum_{i = 1}^n \psi((p e_i) \theta_{ij}(x))\\
& = & \sum_{i = 1}^n \psi(p e_i)\theta_{ij}(x) \: (\psi \mbox{ is a right }B\mbox{-module isomorphism})\\
& = & \sum_{i = 1}^n \xi_i \theta_{ij}(x).
\end{eqnarray*}
Consequently, for every $x \in P$, $x \xi = \xi \theta(x)$. In the von Neumann algebra $\mathcal{M}^{n + 1} \subset B(L^2(\mathcal{M}) \oplus L^2(\mathcal{M})^{\oplus n})$, define
\begin{equation*}
X_x = 
\begin{pmatrix}
x & 0 \\
0 & \theta(x)
\end{pmatrix}, \forall x \in P. 
\end{equation*}
In the space $L^2(\mathcal{M}^{n + 1})$, define
\begin{equation*}
\Xi = \begin{pmatrix}
0 & \xi \\
0 & 0
\end{pmatrix}.
\end{equation*}
We still denote by $1_P$ the unit of $P^{n + 1} \subset \mathcal{M}^{n + 1}$. Note that $X_x \in 1_P \mathcal{M}^{n + 1} 1_P$, $\forall x \in P$, and $\Xi \in 1_P L^2(\mathcal{M}^{n + 1})$. We obtain $X_x \Xi = \Xi X_x$, for every $x \in P$. Since $\Xi$ is a $\gamma$-eigenvector in $\mathcal{M}^{n + 1}$ for the state $\tr_{n + 1} \otimes \varphi$, we can define (as in the Appendix Proposition \ref{polar}) $T_\Xi$ and write $T_\Xi = V |T_\Xi|$ the polar decomposition of $T_\Xi$. We get $X_x V = V X_x$, for every $x \in P$, and $VV^* \leq 1_P$. Write
\begin{equation*}
V = 
\begin{pmatrix}
u & v \\
v' & w
\end{pmatrix}.
\end{equation*}
It is straightforward to check that $v \in M_{1, n}(\C) \otimes 1_P \mathcal{M}$ is a partial isometry from $\ker w$ onto $\ker u^*$ such that $x v = v \theta(x)$, for every $x \in P$. Moreover, $v$ is a $\gamma$-eigenvector for $\varphi$ and $v^*v \leq p$.
\end{proof}

As a consequence of Definition $\ref{embed}$ and Theorem $\ref{newembed}$, we prove the following proposition which will be needed in the next section.

\begin{prop}\label{embed2}
Let $(\mathcal{M}, \varphi)$ be a von Neumann algebra endowed with a faithful normal almost periodic state. Let $\mathcal{M}_i \subset \mathcal{M}$, $(i = 1, 2)$, be two von Neumann subalgebras globally invariant under the modular group $(\sigma^\varphi_t)$. Let $Q \subset \mathcal{M}^\varphi$ be a possibly non-unital von Neumann subalgebra, and denote by $1_Q$ its unit. Assume that $\displaystyle{Q \mathop{\nprec}_{\mathcal{M}} \mathcal{M}_1}$ and $\displaystyle{Q \mathop{\nprec}_{\mathcal{M}} \mathcal{M}_2}$. Then, there exists a sequence of unitaries $(u_k)$ in $Q$ such that $\|E_{\mathcal{M}_i}(x^* u_k y)\|_2 \to 0$ for all $x, y \in 1_Q \mathcal{M}$ and all $i \in \{1, 2\}$.
\end{prop}

\begin{proof}
Denote by
\begin{equation*}
\mathcal{B} = \begin{pmatrix}
\mathcal{M}_1 & 0\\
0 & \mathcal{M}_2
\end{pmatrix} \subset M_2(\C) \otimes \mathcal{M},
\end{equation*}
and $B = \mathcal{B}^{\tr_2 \otimes \varphi}$. Define $\rho : Q \to M_2(\C) \otimes \mathcal{M}$ in the following way
\begin{equation*}
\rho(x) = {\begin{pmatrix}
x & 0\\
0 & x
\end{pmatrix}}, \forall x \in Q.
\end{equation*}
We still denote by $1_Q$ the unit of $M_2(\C) \otimes Q \subset M_2(\C) \otimes \mathcal{M}$. Assume that there is no sequence of unitaries $(u_k)$ in $Q$ such that for all $x, y \in 1_Q \mathcal{M}$ and all $i \in \{1, 2\}$, $\|E_{\mathcal{M}_i}(x^* u_k y)\|_2 \to~0$. It is equivalent to saying that there is no sequence of unitaries $(V_k)$ in $\rho(Q)$ such that $\|E_{\mathcal{B}}(X^* V_k Y)\|_{\tr_2 \otimes \varphi} \to 0$ for all $X, Y \in 1_Q (M_2(\C) \otimes \mathcal{M})$. Using our notation, we get
\begin{equation*}
\rho(Q) \mathop{\prec}_{M_2(\C) \otimes \mathcal{M}} \mathcal{B}.
\end{equation*}
Combining Theorem $\ref{newembed}$ and Definition $\ref{embed}$ (second point), we know that there exists $n \geq 1$, there exist a nonzero element $W$ in $1_Q (M_2(\C) \otimes \mathcal{M})$ and finitely many $W_1, \dots, W_n \in 1_Q (M_2(\C) \otimes \mathcal{M})$, such that $\rho(Q) W \subset \sum_{k = 1}^n W_k B$. Write
\begin{equation*}
W = \begin{pmatrix}
w^a & w^b\\
w^c & w^d
\end{pmatrix}, 
W_k = \begin{pmatrix}
w_k^a & w_k^b\\
w_k^c & w_k^d
\end{pmatrix}.
\end{equation*}
Thus, we obtain
\begin{equation*}
\begin{pmatrix}
Q w^a & Q w^b\\
Q w^c & Q w^d
\end{pmatrix} 
\subset
\begin{pmatrix}
\sum_{k = 1}^n w_k^a \mathcal{M}_1^{\varphi_1} & \sum_{k = 1}^n w_k^b \mathcal{M}_2^{\varphi_2} \\
\sum_{k = 1}^n w_k^c \mathcal{M}_1^{\varphi_1} & \sum_{k = 1}^n w_k^d \mathcal{M}_2^{\varphi_2}
\end{pmatrix}.
\end{equation*}
Since $W \neq 0$, there exists a letter $z \in \{a, b, c, d\}$ such that $w^z \neq 0$. So, there exists $i \in \{1, 2\}$, such that $Q w^z \subset \sum_{k = 1}^n w_k^z \mathcal{M}_i^{\varphi_i}$, and $w^z, w_1^z, \dots, w_n^z \in 1_Q \mathcal{M}$. Thus, combining once again Definition $\ref{embed}$ (second point) and Theorem $\ref{newembed}$, we have proven that there exists $i \in \{1, 2\}$ such that $\displaystyle{Q \mathop{\prec}_{\mathcal{M}} \mathcal{M}_i}$.
\end{proof}

\subsection{Controlling Quasi-Normalizers of Subalgebras of Free Products with Amalgamation.}

Let $Q \subset \mathcal{M}$ be a von Neumann subalgebra of $\mathcal{M}$. An element $x \in \mathcal{M}$ is said to \emph{quasi-normalize} $Q$ \emph{inside} $\mathcal{M}$ if there exist $x_1, \dots, x_k$ and $y_1, \dots, y_r$ in $\mathcal{M}$ such that
\begin{equation*}
xQ \subset \sum_{i = 1}^{k} Qx_i \, \mbox{ and } \, Qx \subset \sum_{j = 1}^{r} y_jQ.
\end{equation*}
The elements quasi-normalizing $Q$ inside $\mathcal{M}$ form a unital $\ast$-subalgebra of $\mathcal{M}$ and their weak closure is called the \emph{quasi-normalizer} of $Q$ inside $\mathcal{M}$. The inclusion $Q \subset \mathcal{M}$ is said to be \emph{quasi-regular} if $\mathcal{M}$ is the quasi-normalizer of $Q$ inside $\mathcal{M}$.

A typical example arises as follows: let $G$ be a countable group and let $H$ be an \emph{almost normal subgroup}, which means that $gHg^{-1} \cap H$ is a finite index subgroup of $H$ for every $g \in G$. It is straightforward to check that the inclusion $L(H) \subset L(G)$ is quasi-regular.

The next result is already known for finite von Neumann algebras: it is a result of Ioana, Peterson \& Popa (see Theorem $1.2.1$ in \cite{ipp}). For our purpose, we need to extend it to von Neumann algebras endowed with almost periodic states.

\begin{theo}\label{amalga}
Let $(\mathcal{M}_1, \varphi_1)$ and $(\mathcal{M}_2, \varphi_2)$ be von Neumann algebras with faithful normal almost periodic states and let $N \subset \mathcal{M}_i^{\varphi_i}$ be a von Neumann subalgebra for $i = 1, 2$. Set $\mathcal{M} = \displaystyle{\mathcal{M}_1 \mathop{\ast}_{N} \mathcal{M}_2}$. Let $Q \subset \mathcal{M}_1^{\varphi_1}$ be a possibly non-unital von Neumann subalgebra, and denote by $1_Q$ its unit. Assume that $\displaystyle{Q \mathop{\nprec}_{\mathcal{M}_1} N}$. Then, every $Q$-$\mathcal{M}_1^{\varphi_1}$ subbimodule $\mathcal{H}$ of $1_Q L^2(\mathcal{M})$ with finite trace over $\mathcal{M}_1^{\varphi_1}$, as a right $\mathcal{M}_1^{\varphi_1}$-module, is contained in $1_Q L^2(\mathcal{M}_1)$. In particular, the quasi-normalizer of $Q$ inside $1_Q \mathcal{M} 1_Q$ is included in $1_Q \mathcal{M}_1 1_Q$, and $1_Q \mathcal{M} 1_Q \cap Q' \subset 1_Q \mathcal{M}_1 1_Q$.
\end{theo}

\begin{proof}
The free product state will be denoted by $\varphi$. Let $\mathcal{A}$ be the linear subspace of $\mathcal{M} \ominus \mathcal{M}_1$ defined by
\begin{equation}\label{densite1}
\mathcal{A} = \span \{ \mathcal{M}_2 \ominus N, (\mathcal{M}_{i_1} \ominus N) \cdots (\mathcal{M}_{i_n} \ominus N) : n \geq 2, i_1 \neq \cdots \neq i_n \in \{1, 2\} \} 
\end{equation}
It is a well-known fact that $\mathcal{A}$ is $\sigma$-weakly dense in $\mathcal{M} \ominus \mathcal{M}_1$. Moreover, since $(\mathcal{M}_i, \varphi_i)$ is almost periodic (for $i = 1, 2$) and $N \subset \mathcal{M}_i^{\varphi_i}$, we have
\begin{equation}\label{densite2}
\mathcal{M}_i \ominus N = \overline{\span}^w \{ \mathcal{M}_i^{\varphi_i} \ominus N,   \mathcal{M}_i^\lambda : \lambda \in \Sp(\mathcal{M}_i, \varphi_i)\backslash\{1\} \}.
\end{equation}
Since $\displaystyle{Q \mathop{\nprec}_{\mathcal{M}_1} N}$, we know from Definition $\ref{embed}$ that there exists a sequence of unitaries $(u_k)$ in $Q$, such that for any $a, b \in 1_Q \mathcal{M}_1$, $\|E_N(a^* u_k b)\|_2 \to 0$.

\begin{claim}\label{esperance}
$\forall x, y \in \mathcal{M} \ominus \mathcal{M}_1, \|E_{\mathcal{M}_1}(x u_k y)\|_2 \to 0$.
\end{claim} 

\begin{proof}[Proof of Claim $\ref{esperance}$]
Let $x$ and $y$ be reduced words in $\mathcal{M}$ with letters alternatingly from $\mathcal{M}_1 \ominus N$ and $\mathcal{M}_2 \ominus N$. We assume that $x$ and $y$ contain at least a letter from $\mathcal{M}_2 \ominus N$. We moreover assume that all the letters of $y$ are eigenvectors for $\varphi$. We set $x = x'a$ with $a = 1$ if $x$ ends with a letter from $\mathcal{M}_2 \ominus N$ and  $a$ equal to the last letter of $x$ otherwise. Note that either $x'$ equals $1$ or is a reduced word ending with a letter from $\mathcal{M}_2 \ominus N$. In the same way, we set $y = by'$ with $b = 1$ if $y$ begins with a letter from $\mathcal{M}_2 \ominus N$ and $b$ equal to the first letter of $y$ otherwise. Note that either $y'$ equals $1$ or is a reduced word beginning with a letter from $\mathcal{M}_2 \ominus N$. Moreover, note that we cannot have at the same time $x' = y' = 1$, and $azb - E_N (azb) \in \mathcal{M}_1 \ominus N$. Then for $z \in Q$,  we have 
\begin{equation*}
E_{\mathcal{M}_1}(x z y) = E_{\mathcal{M}_1}(x' E_N(a z b) y').
\end{equation*}
Since all the letters of $y$ are eigenvectors for $\varphi$, there exists $\lambda > 0$ such that $y' \in \mathcal{M}^\lambda$.  Then 
\begin{eqnarray*}
\|E_{\mathcal{M}_1}(x z y)\|_2 & \leq & \|x' E_N(a z b) y'\|_2\\
& \leq & \lambda^{-1/2}\|x'\| \|y'\| \|E_N(a z b)\|_2.
\end{eqnarray*}
It follows that $\|E_{\mathcal{M}_1}(x u_k y)\|_2 \to 0$. More generally, with the same $y$, for any $x \in \mathcal{A}$, we have $\|E_{\mathcal{M}_1}(x u_k y)\|_2 \to 0$. 

We keep the same $y$, but now we take $x \in \mathcal{M} \ominus \mathcal{M}_1$. We can find a  sequence $(x_i)$ in $\mathcal{A}$ such that $\|x - x_i\|_2 \to 0$. Since $u_k \in Q \subset \mathcal{M}_1^{\varphi_1}$, it follows that $u_k y \in \mathcal{M}^\lambda$, $\forall n \in \N$. We get
\begin{eqnarray*}
\|(x - x_i) u_n y\|_2 & = & \|R_{u_n y} (\widehat{x - x_i})\|_2 \\
& \leq & \lambda^{-1/2} \|u_n y\| \|x - x_i\|_2 \\
& \leq & \lambda^{-1/2} \| y\| \|x - x_i\|_2 
\end{eqnarray*}
Take now $\varepsilon > 0$. Choose $i$ such that $\|x - x_i\|_2 \leq \varepsilon/(2 \lambda^{-1/2}\|y\|)$. Choose now $k_0 \in \N$, such that for any $k \geq k_0$, $\|E_{\mathcal{M}_1}(x_i u_k y)\|_2 \leq \varepsilon/2$. Write $E_{\mathcal{M}_1}(x u_k y) = E_{\mathcal{M}_1}((x - x_i) u_k y) + E_{\mathcal{M}_1}(x_i u_k y)$. For any $n \geq n_0$, we get
\begin{eqnarray*}
\|E_{\mathcal{M}_1}(x u_k y)\|_2 & \leq & \|E_{\mathcal{M}_1}((x - x_i) u_k y)\|_2 + \|E_{\mathcal{M}_1}(x_i u_k y)\|_2\\
& \leq & \|(x - x_i) u_k y)\|_2 + \|E_{\mathcal{M}_1}(x_i u_k y)\|_2\\
& \leq & \varepsilon.
\end{eqnarray*}
Denote by $\mathcal{E}$ the linear span of the $y$'s which are reduced words in $\mathcal{M}$ containing at least a letter from $\mathcal{M}_2 \ominus N$ and such that all the letters of  $y$ are eigenvectors for $\varphi$. We finally get that for any $x \in \mathcal{M} \ominus \mathcal{M}_1$ and any $y \in \mathcal{E}$, $\|E_{\mathcal{M}_1}(x u_k y)\|_2 \to 0$. Note that from $(\ref{densite1})$ and $(\ref{densite2})$, it is straightforward to check that $\mathcal{E}$ is ultraweakly dense in $\mathcal{M} \ominus \mathcal{M}_1$. 

At last, take $x, y \in \mathcal{M} \ominus \mathcal{M}_1$. As before, take $\varepsilon > 0$ and choose $z \in \mathcal{E}$, such that $\|x u_k (y - z)\|_2 \leq \varepsilon/2$, uniformly in $k \in \N$. Choose now $k_0 \in \N$, such that  for any $k \geq k_0$, $\|E_{\mathcal{M}_1}(x u_k z)\|_2 \leq \varepsilon/2$. Write $E_{\mathcal{M}_1}(x u_k y) = E_{\mathcal{M}_1}(x u_k (y - z)) + E_{\mathcal{M}_1}(x u_k z)$. For any $k \geq k_0$, we get
\begin{eqnarray*}
\|E_{\mathcal{M}_1}(x u_k y)\|_2 & \leq & \|E_{\mathcal{M}_1}(x u_k (y - z))\|_2 + \|E_{\mathcal{M}_1}(x u_k z)\|_2\\
& \leq & \|x u_k (y - z)\|_2 + \|E_{\mathcal{M}_1}(x u_k z)\|_2\\
& \leq & \varepsilon.
\end{eqnarray*}
Consequently, for any $x, y \in \mathcal{M} \ominus \mathcal{M}_1$, $\|E_{\mathcal{M}_1}(x u_k y)\|_2 \to 0$. The claim is proven.
\end{proof}

Let $\mathcal{H}$ be a $Q$-$\mathcal{M}_1^{\varphi_1}$-subbimodule of $1_Q L^2(\mathcal{M})$ with finite trace over $\mathcal{M}_1^{\varphi_1}$. Since $\varphi$ is almost periodic, we can write 
\begin{equation*}
\mathcal{H} = \bigoplus_{\gamma \in \Sp(\mathcal{M}, \varphi)} \mathcal{H}^\gamma
\end{equation*}
where all the elements of $\mathcal{H}^\gamma$ are $\gamma$-eigenvectors for $\varphi$. Note that $\mathcal{H}^\gamma$ is nothing but $p_\gamma \mathcal{H}$. Since $Q, \mathcal{M}_1^{\varphi_1} \subset \mathcal{M}^\varphi$, it follows that $p_\gamma \in \langle \mathcal{M}, e_{\mathcal{M}_1^{\varphi_1}} \rangle \cap Q'$ and thus, each of the $\mathcal{H}^\gamma$'s is a $Q$-$\mathcal{M}_1^{\varphi_1}$-subbimodule of $1_Q L^2(\mathcal{M})$ with finite trace over $\mathcal{M}_1^{\varphi_1}$, as a right $\mathcal{M}_1^{\varphi_1}$-module. So, we can assume that $\mathcal{H} = \mathcal{H}^\gamma$ for some $\gamma \in \Sp(\mathcal{M}, \varphi)$. From Lemma $\ref{finiment}$, we know that there exists a sequence $(z_k)$ of central projections in $\mathcal{M}_1^{\varphi_1}$ such that $\mathcal{H}z_k$ is finitely generated as a right $\mathcal{M}_1^{\varphi_1}$-module and $\varphi(z_k) \to 1$. If we prove that $\mathcal{H}z_k \subset 1_Q L^2(\mathcal{M}_1)$, $\forall k \in \N$, we are done. Indeed, assume that $\mathcal{H}z_k \subset 1_Q L^2(\mathcal{M}_1)$, $\forall k \in \N$. Since $\varphi(z_k) \to 1$, it follows that $z_k \to 1$ strongly. Thus, $\forall \xi \in \mathcal{H}$, $\xi = \lim_{k \to \infty} \xi z_k \in 1_Q L^2(\mathcal{M}_1)$. Consequently, $\mathcal{H} \subset 1_Q L^2(\mathcal{M}_1)$.

From now on, we assume that $\mathcal{H} \subset 1_Q L^2(\mathcal{M}^\gamma)$, and $\mathcal{H}$ is finitely generated as a right $\mathcal{M}_1^{\varphi_1}$-module. Then, there exist $n \geq 1$, a projection $p \in (\mathcal{M}_1^{\varphi_1})^n$ and a right $\mathcal{M}_1^{\varphi_1}$-module isomorphism
\begin{equation*}
\psi : pL^2(\mathcal{M}_1^{\varphi_1})^{\oplus n} \to \mathcal{H}.
\end{equation*}
Since $\mathcal{H}$ is a left $Q$-module, there exists a (unital) $\ast$-homomorphism $\theta : Q \to p (\mathcal{M}_1^{\varphi_1})^n p$ such that for every $\eta \in pL^2(\mathcal{M}_1^{\varphi_1})^{\oplus n}$, and every $x \in Q$, $x \psi(\eta) = \psi(\theta(x) \eta)$. For $i \in \{1, \dots, n\}$, let $e_i = (0, \dots, \widehat{1}, \dots, 0) \in L^2(\mathcal{M}_1^{\varphi_1})^{\oplus n}$ and $\xi_i = \psi(p e_i)$. Let $\xi = (\xi_1, \dots, \xi_n) \in M_{1, n}(\C) \otimes \mathcal{H}$. As in Theorem $\ref{newembed}$, we can prove that $x \xi = \xi \theta(x)$, for every $x \in Q$. In the von Neumann algebra $\mathcal{M}^{n + 1} \subset B(L^2(\mathcal{M}) \oplus L^2(\mathcal{M})^{\oplus n})$, define as before
\begin{equation*}
X_x = 
\begin{pmatrix}
x & 0 \\
0 & \theta(x)
\end{pmatrix}, \forall x \in Q. 
\end{equation*}
In the space $L^2(\mathcal{M}^{n + 1})$, define
\begin{equation*}
\Xi = \begin{pmatrix}
0 & \xi \\
0 & 0
\end{pmatrix}.
\end{equation*}
Thus, we obtain $X_x \Xi = \Xi X_x$, for every $x \in Q$. We still denote by $1_Q$ the unit of $Q^{n + 1} \subset \mathcal{M}^{n + 1}$. Since $\Xi$ is a $\gamma$-eigenvector in $\mathcal{M}^{n + 1}$ for the state $\tr_{n + 1} \otimes \varphi$, we can define as before $T_\Xi$ and write $T_\Xi = V |T_\Xi|$ the polar decomposition of $T_\Xi$. We get $X_x V = V X_x$, for every $x \in Q$. Let $f : \R_+ \to \C$ be a bounded Borel function with compact support. By functional calculus, $T_\Xi f(|T_\Xi|) \in 1_Q \mathcal{M}^{n + 1}$. Write 
\begin{equation*}
T_\Xi f(|T_\Xi|) = 
\begin{pmatrix}
y & a \\
a' & z
\end{pmatrix},
\end{equation*} 
with $a \in M_{1, n}(\C) \otimes 1_Q \mathcal{M}$. It is straightforward to check that $x a = a \theta(x)$, for every $x \in Q$. Since $Q, \mathcal{M}_1^{\varphi_1} \subset \mathcal{M}_1$, $x (1 \otimes E_{\mathcal{M}_1})(a) = (1 \otimes E_{\mathcal{M}_1})(a) \theta(x)$, for every $x \in Q$. Write $b = a - (1 \otimes E_{\mathcal{M}_1})(a)$. Note that $b \in M_{1, n}(\C) \otimes 1_Q \mathcal{M}$ and $(1 \otimes E_{\mathcal{M}_1})(b) = 0$. We have $x b = b \theta(x)$, for every $x \in Q$ (note that $\range (b^*b) \subset \range p$). Since $(1 \otimes E_{\mathcal{M}_1})(b) = 0$, we have $\|(\id \otimes E_{\mathcal{M}_1})(b^*u_k b)\|_{\tr_n \otimes \varphi} \to 0$, thanks to Claim $\ref{esperance}$. Since $b^*u_kb = \theta(u_k) b^*b$ and $\theta(u_k) \in \mathcal{U}(p(\mathcal{M}_1^{\varphi_1})^n p)$, we get  
\begin{eqnarray*}
\|(\id \otimes E_{\mathcal{M}_1})(b^*b)\|_{\tr_n \otimes \varphi} & = & \|\theta(u_k) (\id \otimes E_{\mathcal{M}_1})(b^*b)\|_{\tr_n \otimes \varphi} \\
& = & \|(\id \otimes E_{\mathcal{M}_1})(\theta(u_k) b^*b)\|_{\tr_n \otimes \varphi} \\
& = & \|(\id \otimes E_{\mathcal{M}_1})(b^*u_k b)\|_{\tr_n \otimes \varphi} \to 0.
\end{eqnarray*}
Consequently, $(\id \otimes E_{\mathcal{M}_1})(b^*b) = 0$ and so $b = 0$. Thus, $a = (1 \otimes E_{\mathcal{M}_1})(a)$ and so $a \in M_{1, n}(\C) \otimes 1_Q \mathcal{M}_1$.

Take now $f_k = \chi_{[0, k]}$, the characteristic function of the interval $[0, k]$, for each $k \in \N^*$ and write
\begin{equation*}
T_\Xi f_k(|T_\Xi|) = 
\begin{pmatrix}
y_k & a_k \\
a'_k & z_k
\end{pmatrix}.
\end{equation*} 
Applying what we have done, we get $a_k \in M_{1, n}(\C) \otimes 1_Q \mathcal{M}_1$, for every $k \geq 1$. Denote by $P_1$ the orthogonal projection from $1_Q L^2(\mathcal{M})$ onto $1_Q L^2(\mathcal{M}_1)$. Then, for any $k \geq 1$, we have $(1 \otimes P_1)\widehat{a_k} = \widehat{a_k}$. Since $\widehat{a_k} \to \xi$, as $k \to \infty$, we get $(1 \otimes P_1)\xi = \xi$, and so $\xi \in M_{1, n}(\C) \otimes 1_Q L^2(\mathcal{M}_1)$. But the $\xi_i$'s generate $\mathcal{H}$ as a right $\mathcal{M}_1^{\varphi_1}$-module. Thus, $\mathcal{H} \subset 1_Q L^2(\mathcal{M}_1)$.

Take now $x \in 1_Q \mathcal{M} 1_Q$ that quasi-normalizes $Q$ inside $1_Q \mathcal{M} 1_Q$. In particular, there exist $r \geq 1$, $y_1, \dots, y_r \in 1_Q\mathcal{M}1_Q$ such that $Q x \subset \sum_{k = 1}^r y_k Q$. Define $\mathcal{H}$ the $Q$-$\mathcal{M}_1^{\varphi_1}$ subbimodule of $1_Q L^2(\mathcal{M})$ by $\mathcal{H} = \overline{Q x \mathcal{M}_1^{\varphi_1}}$. Since $\mathcal{H} \subset \sum_{k = 1}^r \overline{y_k \mathcal{M}_1^{\varphi_1}}$, it follows that $\mathcal{H}$ is of finite trace over $\mathcal{M}_1^{\varphi_1}$ as a right $\mathcal{M}_1^{\varphi_1}$-module. Thus, $\mathcal{H} \subset 1_Q L^2(\mathcal{M}_1)$. So, $x \in \mathcal{M}_1$. Consequently, the quasi-normalizer of $Q$ inside $\mathcal{M}$ is included in $1_Q \mathcal{M}_1 1_Q$. Obviously, we get also $1_Q \mathcal{M} 1_Q \cap Q' \subset 1_Q \mathcal{M}_1 1_Q$.
\end{proof}

\section{Type ${\rm II_1}$ Factors with Prescribed Countable Fundamental Group}\label{fun}

\subsection{Intertwining Rigid Subalgebras of Crossed Products}

\begin{nota}\label{interpretation}
\emph{Let $\sigma : G \to \Aut(\mathcal{N}, \varphi)$ be a state-preserving action, with $\varphi$ an almost periodic state. We adopt the following notation: 
\begin{enumerate}
\item $\mathcal{M} = \mathcal{N} \rtimes G$, with the action $(\sigma_g)$.
\item $\mathcal{M}_1 = (\mathcal{N} \ast \C) \rtimes G$, with the action $(\sigma_g \ast \id)$.
\item $\mathcal{M}_2 = (\C \ast \mathcal{N}) \rtimes G$, with the action $(\id \ast \sigma_g)$.
\item $\widetilde{\mathcal{M}} = (\mathcal{N} \ast \mathcal{N}) \rtimes G$, with the diagonal action $(\sigma_g \ast \sigma_g)$.
\item $N = \mathcal{N}^\varphi$ and $\widetilde{N} = (\mathcal{N} \ast \mathcal{N})^{\varphi \ast \varphi}$.
\item $M = N \rtimes G$ and $\widetilde{M} = \widetilde{N} \rtimes G$.
\end{enumerate}
It is clear that $M = \mathcal{M}^\varphi$ and $\widetilde{M} = \widetilde{\mathcal{M}}^{\varphi \ast \varphi}$. We shall identify $\mathcal{M}$ with $\mathcal{M}_1$. We regard $\mathcal{M}_{1, 2}$ as  subalgebras of $\widetilde{\mathcal{M}}$ by considering $\mathcal{N} \ast \C$ and $\C \ast \mathcal{N} \subset \mathcal{N} \ast \mathcal{N}$. Moreover, canonically we have the following isomorphism:}
\begin{equation*}
\widetilde{\mathcal{M}} \cong \mathcal{M}_1 \mathop{\ast}_{L(G)} \mathcal{M}_2.
\end{equation*}
\end{nota}

The next theorem is an analogue of a result by Popa (see Theorem $4.4$ in \cite{popamal1}). In the context of {\it free} malleable actions, a {\it gauged extension} for the action $\sigma$ (see Section $1$ in \cite{popamal1}) no longer makes sense. However, regarding a crossed product as a free product with amalgamation (Notation $\ref{interpretation}$), and using {\it free etymology} techniques as in the proof of Theorem $\ref{amalga}$, we are able to prove the following result.

\begin{theo}\label{rigidembed}
Let $\sigma : G \to \Aut(\mathcal{N}, \varphi)$ be a state-preserving \emph{s-malleable} (freely) mixing action. We shall freely use Notation $\ref{interpretation}$. Let $Q \subset M$ be a \emph{diffuse} subalgebra with the \emph{relative property} $(T)$. Denote by $P$ the quasi-normalizer of $Q$ inside $M$.

Then, there exist $\gamma > 0$, $n \geq 1$ and a nonzero partial isometry $v \in M_{1, n}(\C) \otimes \mathcal{M}$ which is a $\gamma$-eigenvector for $\varphi$ and satisfies
\begin{equation*}
vv^* \in P \cap Q', \; v^*v \in L(G)^n, \; v^*Qv \subset v^*Pv \subset v^*v (M_n(\C) \otimes L(G)) v^*v.
\end{equation*}
\end{theo} 
 
\begin{proof} We take $(\alpha_t)$ and $\beta$ as in Definition $\ref{malleable}$. We extend $(\alpha_t)$ and $\beta$ to $\widetilde{\mathcal{M}}$.

{\bf Step (1): Using the relative property $(T)$.} For every $t \in \R$, we have the following $Q$-$Q$-bimodule $\mathcal{H}_t = L^2(\widetilde{M})$, with
\begin{eqnarray*}
x \cdot \xi & = & x\xi \\
\xi \cdot x & = & \xi \alpha_t(x),
\end{eqnarray*}
for all $x \in Q$, $\xi \in L^2(\widetilde{M})$. Since the action $(\alpha_t)$ is continuous, we have $\mathcal{H}_t \to \mathcal{H}_0$ as $t \to 0$, in the sense of correspondences. The relative property $(T)$ yields $t = 2^{-s}$, $s \in \N^*$ and $\xi \in \mathcal{H}_t$, $\xi \neq 0$,  such that
\begin{equation*}
x \xi = \xi \alpha_t(x), \forall x \in Q.
\end{equation*} 
Taking the polar decomposition of the vector $\xi$ (see Proposition \ref{polar}), we find a nonzero partial isometry $v \in \widetilde{M}$ satisfying
\begin{equation}\label{equamal}
xv = v\alpha_t(x), \forall x \in Q.
\end{equation}

{\bf Step (2): Proving $\displaystyle{Q \mathop{\prec}_{\mathcal{M}_1} L(G)}$ using the amalgamation over $L(G)$.} Assume that $\displaystyle{Q \mathop{\nprec}_{\mathcal{M}_1} L(G)}$. We shall obtain a contradiction. Definition $\ref{embed}$ yields a sequence of unitaries $(u_k)$ in $Q$ such that for any $a, b \in \mathcal{M}_1$, $\|E_{L(G)}(a u_k b)\|_2 \to 0$.

First of all, we shall find a nonzero partial isometry in $\widetilde{M}$ satisfying Equation $(\ref{equamal})$ for $t = 1$. In order to do so, it suffices to prove the existence of a nonzero partial isometry $w \in \widetilde{M}$ satisfying $xw = w\alpha_{2t}(x)$ for all $x \in Q$. Indeed, iterating the procedure then allows to continue till $t = 1$. Thanks to Theorem $\ref{amalga}$, with $N = L(G)$, we get $\widetilde{\mathcal{M}} \cap Q' \subset \mathcal{M}_1$ and so $\widetilde{M} \cap Q' \subset M$. In particular, $vv^* \in M$. We write $vv^* = p$, with $p \in M$. Using the properties of $\beta$ (in particular, $\beta(x) = x$ for all $x \in M$) and Equation $(\ref{equamal})$, one checks that $w := \alpha_t(\beta(v^*)v)$ is an element of $\widetilde{M}$ satisfying $xw = w\alpha_{2t}(x)$ for all $x \in Q$. Indeed, for any $x \in Q$,
\begin{eqnarray*}
w\alpha_{2t}(x) & =&  \alpha_t(\beta(v^*)v\alpha_t(x)) \\
& = & \alpha_t(\beta(v^*)xv) \\
& = & \alpha_t(\beta(v^*x) v) \\
& = & \alpha_t(\beta(\alpha_t(x)v^*)v) \\
& = & \alpha_t \beta \alpha_t(x) \alpha_t(\beta(v^*)v) \\
& = & \beta(x)w \\
& = & xw.
\end{eqnarray*}
Moreover,
\begin{equation*}
ww^* = \alpha_t(\beta(v^*)p\beta(v)) = \alpha_t(\beta(v^*pv)) = \alpha_t(\beta(v^*v)).
\end{equation*}
The last term is a nonzero projection. So, $w$ is the required nonzero partial isometry. Thus, we have found a nonzero partial isometry $v \in \widetilde{M}$ satisfying
\begin{equation*}
xv = v\alpha_1(x), \forall x \in Q.
\end{equation*}
 
Observe now that using the second point of Definition $\ref{embed}$, we get $\displaystyle{\alpha_1(Q) \mathop{\nprec}_{\mathcal{M}_2} L(G)}$, since $\alpha_1(L(G)) = L(G)$ and $\alpha_1(\mathcal{M}_1) = \mathcal{M}_2$. Thus, by Theorem $\ref{amalga}$, we get $\widetilde{\mathcal{M}} \cap \alpha_1(Q)' \subset \mathcal{M}_2$ and so $\widetilde{M} \cap \alpha_1(Q)' \subset \alpha_1(M)$. In particular, $v^*v \in \alpha_1(M)$. 

\begin{claim}\label{esperance2}
$\forall x, y \in \widetilde{\mathcal{M}}, \| E_{\mathcal{M}_2}(x u_k y) \|_2 \to 0$.
\end{claim}

\begin{proof}[Proof of Claim $\ref{esperance2}$]
Regarding $\displaystyle{\widetilde{\mathcal{M}} = \mathcal{M}_1 \mathop{\ast}_{L(G)} \mathcal{M}_2}$, let $x, y \in \widetilde{\mathcal{M}}$ be either in $L(G)$ or reduced words in $\widetilde{\mathcal{M}}$ with letters alternatingly from $\mathcal{M}_1 \ominus L(G)$ and $\mathcal{M}_2 \ominus L(G)$. We assume as in the proof of Claim $\ref{esperance}$ that all the letters of $y$ are eigenvectors for $\varphi$. We set $x = x'a$ with $a = x$ if $x \in L(G)$, $a = 1$ if $x$ ends with a letter from $\mathcal{M}_2 \ominus L(G)$ and  $a$ equal to the last letter of $x$ otherwise. Note that $x'$ is either equal to $1$ or a reduced word ending with a letter from $\mathcal{M}_2 \ominus L(G)$. In the same way, we set $y = by'$ with $b = y$ if $y \in L(G)$, $b = 1$ if $y$ begins with a letter from $\mathcal{M}_2 \ominus L(G)$ and $b$ equal to the first letter of $y$ otherwise. Note that $y'$ is either equal to $1$ or a reduced word beginning with a letter from $\mathcal{M}_2 \ominus L(G)$. Then for $z \in Q$,  we have $azb - E_{L(G)}(azb) \in \mathcal{M}_1 \ominus L(G)$, and thus
\begin{equation*}
E_{\mathcal{M}_2}(x z y) = E_{\mathcal{M}_2}(x' E_{L(G)}(a z b) y').
\end{equation*}
Since all the letters of $y$ are eigenvectors for $\varphi$, there exists $\lambda > 0$ such that $y' \in \widetilde{\mathcal{M}}^\lambda$. Therefore,
\begin{eqnarray*}
\|E_{\mathcal{M}_2}(x z y)\|_2 & \leq & \| x' E_{L(G)}(a z b) y'\|_2\\
& \leq & \lambda^{-1/2} \|x'\| \|y'\| \| E_{L(G)}(a z b)\|_2.
\end{eqnarray*}
It follows that $\|E_{\mathcal{M}_2}(x u_k y)\|_2 \to 0$. We can proceed exactly the way we did in the proof of Claim $\ref{esperance}$, in order to obtain that $\|E_{\mathcal{M}_2}(x u_k y)\|_2 \to 0$, for every $x, y \in \widetilde{\mathcal{M}}$.
\end{proof}

We remind that for any $x \in Q$, $v^*xv = \alpha_1(x)v^*v$. Moreover, $v^*v \in\alpha_1(M) \subset \mathcal{M}_2$. So, for any $x \in Q$, $v^*xv \in \mathcal{M}_2$. Since $\alpha_1(u_k) \in \mathcal{U}(\mathcal{M}_2)$, we get
\begin{equation*}
\|v^*v\|_2 = \|\alpha_1(u_k)v^*v\|_2 = \|E_{\mathcal{M}_2}(\alpha_1(u_k)v^*v)\|_2 = \|E_{\mathcal{M}_2}(v^*u_kv)\|_2 \to 0.
\end{equation*}
Thus $v = 0$, which is a contradiction.

{\bf Step (3): Using the mixing property of the action to conclude.} 
From Definition $\ref{embed}$, we get $\gamma > 0$, $n \geq 1$, $p$ a projection in $L(G)^n$, a (unital) $\ast$-homomorphism $\theta : Q \to pL(G)^np$ and a nonzero partial isometry $w \in M_{1, n}(\C) \otimes \mathcal{M}$ such that $w$ is a $\gamma$-eigenvector for $\varphi$ and $xw = w\theta(x)$ for all $x \in Q$. It follows that $w^*w \in pM^np \cap \theta(Q)'$. Since $\theta(Q)$ is diffuse and since the action is mixing, the quasi-normalizer of $\theta(Q)$ inside $pM^np$ is included in $pL(G)^np$ by Theorem $3.1$ of \cite{popamal1} (see also Theorem D.4 of \cite{vaesbern}). Take now $x$ that quasi-normalizes $Q$ inside $M$. Thus, there exist $x_1, \dots, x_k$ and $y_1, \dots, y_r$ in $M$ such that
\begin{equation*}
xQ \subset \sum_{i = 1}^{k} Qx_i \, \mbox{ and } \, Qx \subset \sum_{j = 1}^{r} y_jQ.
\end{equation*}
Observe moreover that  $ww^* \in M \cap Q'$. We get
\begin{eqnarray*}
\theta(Q)w^*xw & \subset & w^*Qxw \\
 & \subset & \sum w^*y_jQw\\
 & \subset & \sum w^*y_jw\theta(Q).
\end{eqnarray*}
Exactly in the same way, we prove that
\begin{equation*}
w^*xw\theta(Q)  \subset  \sum \theta(Q)w^*x_iw.
\end{equation*}
Thus $w^*Pw$ is included in the quasi-normalizer of $\theta(Q)$ inside $pM^np$, and so $w^*Pw \subset pL(G)^np$. Since obviously, $M \cap Q' \subset P$, we can take $v = w$ to conclude. 
\end{proof}

\begin{rem}\label{remimportante}
\emph{Note that we used the period 2 automorphism $\beta$ in a very crucial way. We do not know if the result still holds true for a \emph{malleable} action, even if $P$ (the quasi-normalizer of $Q$ in $M = N \rtimes G$) is assumed to be a factor. Remark $4.5$ in \cite{ioana3} may shed light on this problem.
}
\end{rem}

\subsection{Intertwining Rigid Subalgebras of Free Products with Amalgamation}

\begin{nota}\label{interpretation2}
\emph{Let $(\mathcal{M}_1, \varphi_1)$ and $(\mathcal{M}_2, \varphi_2)$ be von Neumann algebras endowed with almost periodic states. Assume that $N \subset \mathcal{M}_i^{\varphi_i}$ for $i = 1, 2$. We write
\begin{enumerate}
\item $\mathcal{M} = \displaystyle{\mathcal{M}_1 \mathop{\ast}_{N} \mathcal{M}_2}$, with $\varphi$ the free product state.
\item $M$ denotes the centralizer of $\mathcal{M}$.
\item $\widetilde{\mathcal{M}}_i = \displaystyle{\mathcal{M}_i \mathop{\ast}_{N} (N \otimes L(\Z))}$, denoting by $u_i \in L(\Z)$ the canonical generating unitary sitting in $\widetilde{\mathcal{M}}_i$.
\item $\widetilde{\mathcal{M}} = \displaystyle{\mathcal{M} \mathop{\ast}_{N} (N \otimes L(\mathbf{F}_2))} = \displaystyle{\widetilde{\mathcal{M}}_1 \mathop{\ast}_{N} \widetilde{\mathcal{M}}_2}$. 
\item $\widetilde{M}$ denotes the centralizer of $\widetilde{\mathcal{M}}$.
\end{enumerate}
}
\end{nota} 
Note that canonically, we have the following isomorphism:
\begin{equation}\label{isoampli}
\mathcal{M}^n \cong \mathcal{M}_1^n \mathop{\ast}_{N^n} \mathcal{M}_2^n, \forall n \in \N^*.
\end{equation}

The next theorem can be viewed as a generalization to the almost periodic case of a result by Ioana, Peterson \& Popa. They proved (see Theorem $0.1$ in \cite{ipp}) that any relatively rigid von Neumann subalgebra $\displaystyle{Q \subset (M_1, \tau_1) \mathop{\ast}_B (M_2, \tau_2)}$ can be intertwined into one of the $M_i$'s. We prove a similar result replacing the faithful normal trace $\tau_{1, 2}$ by any almost periodic faithful normal state $\varphi_{1,2}$. The beautiful idea of the proof of Theorem $\ref{embedamalga}$ was given to us by Stefaan Vaes. We gratefully thank him for allowing us to present it here.

\begin{theo}\label{embedamalga}
Let $(\mathcal{M}_1, \varphi_1)$ and $(\mathcal{M}_2, \varphi_2)$ be von Neumann algebras with faithful normal almost periodic states and $N \subset \mathcal{M}_i^{\varphi_i}$ for $i = 1, 2$. Set $\mathcal{M} = \displaystyle{\mathcal{M}_1 \mathop{\ast}_{N} \mathcal{M}_2}$ and $\varphi$ the free product state. We shall freely use Notation $\ref{interpretation2}$. If $Q \subset \mathcal{M}^{\varphi}$ is rigid, then there exists $i \in \{1, 2\}$ such that $\displaystyle{Q \mathop{\prec}_{\mathcal{M}} \mathcal{M}_i}$.
\end{theo} 
 
\begin{proof}
We assume that $\displaystyle{Q \mathop{\nprec}_{\mathcal{M}} \mathcal{M}_1}$ and $\displaystyle{Q \mathop{\nprec}_{\mathcal{M}} \mathcal{M}_2}$; we shall obtain a contradiction.

{\bf Step (0): Defining the deformation property.}
Consider $L(\mathbf{F}_2) = L(\Z) \ast L(\Z)$ with its canonical unitaries $u_1$ and $u_2$. Let $f : \mathbf{S}^1 \to ]-\pi, \pi]$ be the Borel function satisfying $\exp(if(z)) = z$ for all $z \in \mathbf{S}^1$. Define the self-adjoint elements $h_i = f(u_i)$ for $i = 1, 2$. Regarding $\widetilde{\mathcal{M}} = \displaystyle{\widetilde{\mathcal{M}}_1 \mathop{\ast}_{N} \widetilde{\mathcal{M}}_2}$, define the one-parameter group of automorphisms $(\alpha_t)$ on $\widetilde{\mathcal{M}}$ by:
\begin{equation*}
\alpha_t = (\Ad \exp(ith_1)) \ast (\Ad \exp(ith_2)).
\end{equation*}
Note that $\alpha_1 = (\Ad u_1) \ast (\Ad u_2)$. Define now the period $2$ automorphism $\beta$ on $L(\mathbf{F}_2)$ by:
\begin{equation*}
\beta(u_1) = u_1^*, \beta(u_2) = u_2^*.
\end{equation*}
Regarding $\widetilde{\mathcal{M}} = \displaystyle{\mathcal{M} \mathop{\ast}_{N} (N \otimes L(\mathbf{F}_2))}$, extend $\beta$ to $\widetilde{\mathcal{M}}$ using the identity automorphism on $\mathcal{M}$. We know from Lemma $2.2.2$ in \cite{ipp} that $\beta\alpha_t = \alpha_{-t}\beta$ for every $t \in \R$. Thus this deformation if of {\it malleable} type as in Definition $\ref{malleable}$.

{\bf Step (1): Using the relative property $(T)$.}
Recall that $\widetilde{M}$ denotes the centralizer of $\widetilde{\mathcal{M}}$ and $M$ the centralizer of $\mathcal{M}$. For every $t \in \R$, define $\mathcal{H}_t = L^2(\widetilde{M})$ the following $Q$-$Q$ bimodule:
\begin{eqnarray*}
x \cdot \xi & = & x\xi, \\
\xi \cdot x & = & \xi\alpha_t(x), \forall x \in Q, \forall \xi \in L^2(\widetilde{M}).
\end{eqnarray*}
Since the action $(\alpha_t)$ is continuous, $\mathcal{H}_t \to \mathcal{H}_0$ as $t \to 0$, in the sense of correspondences. Thus the relative property $(T)$ (and Proposition \ref{polar}) yields $t = 2^{-s}$, $s \in \N$ and a nonzero partial isometry $v \in \widetilde{M}$ satisfying
\begin{equation}\label{equamal2}
xv = v\alpha_t(x), \forall x \in Q.
\end{equation}

{\bf Step (2): Going till $t = 1$ using the deformation property.}
We shall find a nonzero partial isometry in $\widetilde{M}$ satisfying Equation $(\ref{equamal2})$ for $t  = 1$. As in the proof of Theorem $\ref{rigidembed}$, in order to do so, it suffices to prove the existence of a nonzero partial isometry $w \in \widetilde{M}$ satisfying $xw = w\alpha_{2t}(x)$ for all $x \in Q$. Indeed, iterating the procedure then allows to continue till $t = 1$.

Since $\displaystyle{Q \mathop{\nprec}_{\mathcal{M}} \mathcal{M}_1}$, certainly $\displaystyle{Q \mathop{\nprec}_{\mathcal{M}} N}$. Regarding $\widetilde{\mathcal{M}} = \displaystyle{\mathcal{M} \mathop{\ast}_{N} (N \otimes L(\mathbf{F}_2))}$, Theorem $\ref{amalga}$ implies that $\widetilde{\mathcal{M}} \cap Q' \subset \mathcal{M}$ and therefore $\widetilde{M} \cap Q' \subset M$. From Equation $(\ref{equamal2})$, we get $vv^* \in \widetilde{M} \cap Q'$, thus $vv^* \in M$. We write $vv^* = p$ with $p \in M$. Using the properties of $\beta$ (in particular, $\beta(x) = x$ for all $x \in \mathcal{M}$), one checks, as in the proof of Theorem $\ref{rigidembed}$, that $w := \alpha_t(\beta(v^*)v)$ is a nonzero partial isometry satisfying $xw = w\alpha_{2t}(x)$ for all $x \in Q$.

{\bf Step (3): Using the amalgamation over $N$ to obtain the contradiction.} We write $\theta = \alpha_1$. We have found a nonzero partial isometry $v \in \widetilde{M}$ satisfying $xv = v\theta(x)$ for all $x \in Q$. Since $\displaystyle{Q \mathop{\nprec}_{\mathcal{M}} \mathcal{M}_1}$, certainly $\displaystyle{\theta(Q) \mathop{\nprec}_{\theta(\mathcal{M})} \theta(\mathcal{M}_1)}$, and thus $\displaystyle{\theta(Q) \mathop{\nprec}_{\theta(\mathcal{M})} N}$. Regarding $\widetilde{\mathcal{M}} = \displaystyle{\mathcal{M} \mathop{\ast}_{N} (N \otimes L(\mathbf{F}_2))}$, Theorem $\ref{amalga}$ implies that $\widetilde{\mathcal{M}} \cap \theta(Q)' \subset \theta(\mathcal{M})$ and therefore $\widetilde{M} \cap \theta(Q)' \subset \theta(M)$. Since $v^*v \in \widetilde{M} \cap \theta(Q)'$, we get $v^*v \in \theta(M)$.

Set $A = L(\F_2)$ and define the subspace $H_{\alt} \subset L^2(\widetilde{\mathcal{M}})$ as the closed linear span of $N$ and the words in $\displaystyle{\mathcal{M}_1 \mathop{\ast}_{N} \mathcal{M}_2 \mathop{\ast}_N (N \otimes A)}$ with letters alternatingly from $\mathcal{M}_1 \ominus N$, $\mathcal{M}_2 \ominus N$, $N \otimes (A \ominus \C1)$ and such that two consecutive letters never come from $\mathcal{M}_1 \ominus N$, $\mathcal{M}_2 \ominus N$. This means that letters from $\mathcal{M}_1 \ominus N$ and $\mathcal{M}_2 \ominus N$ are always separated by a letter from $N \otimes (A \ominus \C1)$.

By the definition of $\theta$, it follows that $\theta(\mathcal{M}) \subset H_{\alt}$. Denote by $P_{\alt}$ the orthogonal projection of $L^2(\widetilde{\mathcal{M}})$ onto $H_{\alt}$. Since $\displaystyle{Q \mathop{\nprec}_{\mathcal{M}} \mathcal{M}_1}$ and $\displaystyle{Q \mathop{\nprec}_{\mathcal{M}} \mathcal{M}_2}$, thanks to Proposition $\ref{embed2}$, we know that there exists a sequence of unitaries $(u_k)$ in $Q$ such that $\|E_{\mathcal{M}_i}(x u_k y)\|_2 \to 0$ for all $x, y \in \mathcal{M}$ and all $i \in \{1, 2\}$. Moreover, we have the following:

\begin{claim}\label{esperance3}
$\forall c, d \in \widetilde{\mathcal{M}}$, $\| P_{\alt}(c u_k d) \|_2 \to 0$.
\end{claim}

\begin{proof}[Proof of Claim $\ref{esperance3}$]
Let $c, d \in \widetilde{\mathcal{M}} = \displaystyle{\mathcal{M} \mathop{\ast}_{N} (N \otimes A)}$ be either in $N$ or reduced words with letters alternatingly from $\mathcal{M} \ominus N$ and $N \otimes (A \ominus \C1)$. As we did in the proof of Claim $\ref{esperance}$, we assume that all the letters of $d$ are eigenvectors for $\varphi$. Set $c = c'a$, with $a = c$ if $c \in N$, $a = 1$ if $c$ ends with a letter from $N \otimes (A \ominus \C1)$ and $a$ equal to the last letter of $c$ otherwise. Note that either $c'$ is equal to $1$ or is a reduced word ending with a letter from $N \otimes (A \ominus \C1)$. Exactly in the same way, set $d = bd'$, with $b = d$ if $d \in N$, $b = 1$ if $d$ begins with a letter from $N \otimes (A \ominus \C1)$ and $b$ equal to the first letter of $d$ otherwise. For $x \in \mathcal{M}$, write $cxd = c' (axb) d'$, and note that $axb \in \mathcal{M}$. Note that either $d'$ is equal to $1$ or is a reduced word beginning with a letter from $N \otimes (A \ominus \C1)$. Since $\mathcal{M} = \displaystyle{\mathcal{M}_1 \mathop{\ast}_N \mathcal{M}_2}$, recall that
\begin{equation*}
\mathcal{M} = \overline{\span}^w \{ N, \mathcal{M}_1 \ominus N, \mathcal{M}_2 \ominus N,  (\mathcal{M}_{i_1} \ominus N) \cdots (\mathcal{M}_{i_n} \ominus N); n \geq 2, i_1 \neq \cdots \neq i_n \in \{1, 2\} \}.
\end{equation*}
By definition of the projection $P_{\alt}$, it is clear that 
\begin{equation}\label{estimee1}
P_{\alt}(c' z d') = 0, \forall z \in \overline{\span}^w \{ (\mathcal{M}_{i_1} \ominus N) \cdots (\mathcal{M}_{i_n} \ominus N); n \geq 2, i_1 \neq \cdots \neq i_n \in \{1, 2\} \}.
\end{equation}
Denote by $P_1$ the orthogonal projection of $L^2(\mathcal{M})$ onto the space $L^2(N) \oplus L^2(\mathcal{M}_1 \ominus N) \oplus L^2(\mathcal{M}_2 \ominus N)$. By definition of the conditional expectations $E_{\mathcal{M}_i}$ $(i = 1, 2)$, it is easy to see that
\begin{eqnarray}\label{estimee2}
P_1(z) & = & E_{\mathcal{M}_1}(z) + E_{\mathcal{M}_2}(z) - E_N(z), \\ \nonumber
\|P_1(z)\|_2^2  & \leq & \|E_{\mathcal{M}_1}(z)\|_2^2 + \|E_{\mathcal{M}_2}(z)\|_2^2, \forall z \in \mathcal{M}.
\end{eqnarray}
We recall that all the letters of $d$ are assumed to be eigenvectors for $\varphi$. Thus, there exists $\lambda > 0$ such that $d'$ is a $\lambda$-eigenvector. Thanks to $(\ref{estimee1})$ and $(\ref{estimee2})$, we get for any $z \in Q$
\begin{eqnarray}\label{estimeeproj}
\|P_{\alt}(c z d)\|_2^2 & = & \|P_{\alt}(c' P_1(azb) d')\|_2^2\\ \nonumber
& \leq & \|c' P_1(azb) d'\|_2^2\\ \nonumber
& \leq & \lambda^{-1} \|c'\|^2 \|d'\|^2 \|P_1(azb)\|_2^2\\ \nonumber
& \leq & \lambda^{-1} \|c'\|^2 \|d'\|^2 \left( \|E_{\mathcal{M}_1}(azb)\|_2^2 + \|E_{\mathcal{M}_2}(azb)\|_2^2 \right).  
\end{eqnarray}

Since $\|E_{\mathcal{M}_i}(x u_k y)\|_2 \to 0$ for all $x, y \in \mathcal{M}$ and all $i \in \{1, 2\}$, the inequality $(\ref{estimeeproj})$ implies that $\|P_{\alt}(c u_k d)\|_2 \to 0$ for $c, d$ chosen as before. We can now proceed exactly the way we did in the proof of Claim $\ref{esperance}$, in order to obtain that $\|P_{\alt}(c u_k d)\|_2 \to 0$, for every $c, d \in \widetilde{\mathcal{M}}$. 
\end{proof}

At last, for any $x \in Q$, $v^*xv = \theta(x) v^*v \in \theta(M) \subset H_{\alt}$. So, since $\theta(u_k) \in \mathcal{U}(M)$, we get
\begin{equation*}
\|v^*v\|_2 = \| \theta(u_k)v^*v \|_2 = \| P_{\alt}(\theta(u_k)v^*v) \|_2 = \| P_{\alt}(v^*u_k v) \|_2 \to 0.
\end{equation*}
It follows that $v = 0$, which is a contradiction.
\end{proof}

\subsection{Fundamental Groups of Type ${\rm II_1}$ Factors}

We denote by $\mathcal{F}(M) \subset \R^*_+$ the fundamental group of a type ${\rm II_1}$ factor $M$, and by $\Sp(\mathcal{N}, \varphi) \subset \R^*_+$ the point spectrum of the modular operator $\Delta_\varphi$ of an almost periodic state $\varphi$ on $\mathcal{N}$. We shall denote by $\Gamma_{\Sp(\mathcal{N}, \varphi)} \subset \R^*_+$ the subgroup generated by $\Sp(\mathcal{N}, \varphi)$. Note that if the centralizer $\mathcal{N}^\varphi$ is a factor, then $\Sp(\mathcal{N}, \varphi)$ is a multiplicative subgroup and then $\Gamma_{\Sp(\mathcal{N}, \varphi)} = \Sp(\mathcal{N}, \varphi)$ (see \cite{dykema96}). As a consequence of Theorem $\ref{rigidembed}$, we obtain the following result.

\begin{theo}\label{rigidaraki}
Let $G$ be an ICC $w$-rigid group. Let $\sigma : G \to \Aut(\mathcal{N}, \varphi)$ be a state-preserving s-malleable (freely) mixing action with $\varphi$ an almost periodic state. Assume that the centralizer $\mathcal{N}^\varphi$ is a factor. Denote by $M$ the crossed product $\mathcal{N}^\varphi \rtimes G$  as in Notation $\ref{interpretation}$. Then, $M$ is a type ${\rm II_1}$ factor and one has
\begin{equation*}
\Sp(\mathcal{N}, \varphi) \subset \mathcal{F}(M) \subset \Sp(\mathcal{N}, \varphi) \mathcal{F}(L(G)).
\end{equation*}
In particular, if $\mathcal{F}(L(G)) = \{1\}$, then $\mathcal{F}(M) = \Sp(\mathcal{N}, \varphi)$.
\end{theo} 
 
\begin{proof}
We refer to the proof of Theorem $5.2$ and Corollary $5.4$ in \cite{popamal1} (see also Theorem $7.1$ in \cite{vaesbern}). The arguments are exactly the same. However, we shall give the proof for the sake of completeness. We should mention here that \emph{controlling} quasi-normalizers will rely on a result of Popa (Theorem $3.1$ in \cite{popamal1}). We remind that Theorem $3.1$ in \cite{popamal1} uses in a crucial way the mixing property of the action. Denote as in Notation \ref{interpretation}, $\mathcal{N} \rtimes G$ by $\mathcal{M}$ and $\mathcal{N}^\varphi \rtimes G$ by $M$. Since the group $G$ is ICC and $\mathcal{N}^\varphi$ is a factor, $M = \mathcal{N}^\varphi \rtimes G$ is necessarily a type ${\rm II_1}$ factor. Note that $\Sp(\mathcal{M}, \varphi) = \Sp(\mathcal{N}, \varphi) \subset \R^*_+$ is a multiplicative subgroup. 

It was shown in \cite{golodets} that the inclusion $\Sp(\mathcal{N}, \varphi) \subset \mathcal{F}(M)$ holds. Indeed, take $\gamma \in \Sp(\mathcal{N}, \varphi)$ and $v$ a nonzero partial isometry in $\mathcal{N}^\gamma$. Write $p = v^*v$ and $q = vv^*$. Then, $p, q \in \mathcal{N}^\varphi \subset M$, $\varphi(q) = \gamma\varphi(p)$, and $\Ad(v)$ yields a $\ast$-isomorphism between $pMp$ and $qMq$. Therefore, $\gamma \in \mathcal{F}(M)$.

Conversely, assume that $t \in \mathcal{F}(M)$ and let $\theta : M \to M^t$ be a $\ast$-isomorphism. We assume that $t \geq 1$. Realize $M^t := p(M_n(\C) \otimes M)p$. Let $H \subset G$ be an infinite normal subgroup with the relative property $(T)$. Since $H \subset G$ is normal, it is clear that $L(G)$ is contained in the quasi-normalizer of $L(H)$ inside $M$. Moreover since $L(H)$ is diffuse, Theorem $3.1$ of \cite{popamal1} implies that the quasi-normalizer of $L(H)$ inside $M$ is exactly $L(G)$. Write $Q = \theta(L(H))$ and $P = \theta(L(G))$. The inclusion $Q \subset P$ still has the relative property $(T)$ and $P$ is the quasi-normalizer of $Q$ inside $M^t$. Since $s = 1/t \leq 1$, choose a projection $q \in Q$ with trace $s$. Write $Q^s := q Q q$ and $P^s := q P q$. We regard $Q^s \subset P^s \subset M$. The inclusion $Q^s \subset P^s$ has the relative property $(T)$, $Q^s$ is diffuse and $P^s$ is the quasi-normalizer of $Q^s$ inside $M$.

We can apply Theorem $\ref{rigidembed}$ in order to obtain $\gamma > 0$ and a nonzero partial isometry $w \in M_{1, r}(\C) \otimes \mathcal{M}$ which is a $\gamma$-eigenvector for $\varphi$ and such that $w^*w \in M_r(\C) \otimes L(G)$, $ww^* \in M \cap (Q^s)' \subset P^s$ and 
\begin{equation*}
w^* Q^s w \subset w^* P^s w \subset w^*w (M_r(\C) \otimes L(G)) w^*w.
\end{equation*}
Since $P$ is a factor, we can find partial isometries $x_1, \dots, x_m \in P$ such that $x_i^* x_i \leq ww^*$, for every $i \in \{1, \dots, m\}$ and $\sum_i x_i x_i^* = 1_P$. Let $k = mr$, and write $v = [x_1 w \cdots x_m w]$. Since $P \subset M^t = p(M_n(\C) \otimes M)p$, we can regard $v \in M_{n, k}(\C) \otimes \mathcal{M}$, and $v$ is a $\gamma$-eigenvector for $\varphi$. Moreover, we have  
\begin{equation}\label{inclusion1}
v^* Q v \subset v^* P v \subset L(G)^{t/\gamma},
\end{equation}
with $vv^* = p$, $v^*v := q \in M_k(\C) \otimes L(G)$, and $L(G)^{t/\gamma} := q (M_k(\C) \otimes L(G)) q$. Note that increasing $n$ or $k$ if necessary, we may assume $k = n$. We want to prove that in fact, $v^* P v = q (M_n(\C) \otimes L(G)) q = L(G)^{t/\gamma}$. If we do so, we are done. Indeed, we have $L(G) \simeq L(G)^{t/\gamma}$, and so $t/\gamma \in \mathcal{F}(L(G))$. Consequently, $t \in \Sp(\mathcal{N}, \varphi)\mathcal{F}(L(G))$.

Changing $q$ to an equivalent projection in $M_n(\C) \otimes L(H)$, we may assume that $q \in M_n(\C) \otimes L(H)$. Define
\begin{equation*}
Q_1 := \theta^{-1}(v (M_n(\C) \otimes L(H)) v^*) \mbox{ and } P_1 := \theta^{-1}(v (M_n(\C) \otimes L(G)) v^*).
\end{equation*}
The inclusion $Q_1 \subset P_1$ has the relative property $(T)$, $P_1$ is the quasi-normalizer of $Q_1$ inside $M$, and Equation $(\ref{inclusion1})$ yields $L(G) \subset P_1$. We want to prove that $P_1 \subset L(G)$. Once again using Theorem $\ref{rigidembed}$, we get that there exist $k \geq 1$, $\lambda > 0$ and a nonzero partial isometry $w \in M_{1, k}(\C) \otimes \mathcal{M}^\lambda$, such that $ww^* = 1$, $w^*w \in M_k(\C) \otimes L(G)$ and $w^* P_1 w \subset L(G)^{1/\lambda}$, where we have realized $L(G)^{1/\lambda} := w^*w (M_k(\C) \otimes L(G)) w^*w$. Since $ww^* = 1$, we have $P_1 w \subset w L(G)^{1/\lambda}$.  Since $L(G) \subset P_1$, it follows that  $L(G) w \subset w L(G)^{1/\lambda}$. From Theorem $3.1$ in \cite{popamal1}, since $L(G)$ is diffuse, we know that any $L(G)$-$L(G)$ subbimodule $\mathcal{H}$ of $L^2(\mathcal{M})$ such that $\dim(\mathcal{H}_{L(G)}) < \infty$ (as a right $L(G)$-module) is contained in $L^2(L(G))$. In particular, this implies that $w \in M_{1, k}(\C) \otimes L(G)$. Since $P_1 \subset w (L(G)^{1/\lambda}) w^*$, we get $P_1 \subset L(G)$. We are done.
\end{proof} 
 
 As a corollary, we obtain the result we mentioned in the introduction.

\begin{cor}
Let $G$ be an ICC $w$-rigid group such that $\mathcal{F}(L(G)) = \{1\}$. Let $\Gamma \subset \R^*_+$ be a countable subgroup. Set $(\mathcal{N}, \varphi) = (T_\Gamma, \varphi_\Gamma)$ the unique almost periodic free Araki-Woods factor whose $\Sd$ invariant equals $\Gamma$. Assume that $G$ acts on $(\mathcal{N}, \varphi)$ by free Bogoliubov shifts w.r.t. to the left regular representation $\lambda_G$ (see Section $\ref{exa}$). Write $M = \mathcal{N}^\varphi \rtimes G$. Then $M$ is a type ${\rm II_1}$ factor and $\mathcal{F}(M) = \Gamma$.
\end{cor}

We prove at last the second result we mentioned in the introduction. A type ${\rm II_1}$ factor $N$ is said to be $w$-\emph{rigid} if it contains a diffuse von Neumann subalgebra $B$ such that the inclusion $B \subset N$ is quasi-regular and has the relative property $(T)$. Of course, if $G$ is an ICC $w$-rigid group, $L(G)$ is a $w$-rigid type ${\rm II_1}$ factor.

\begin{theo}\label{generalresult}
Let $N$ be a $w$-rigid type ${\rm II_1}$ factor such that $\mathcal{F}(N) = \{1\}$. Let $(\mathcal{A}, \psi)$ be a von Neumann algebra endowed with an almost periodic state. Assume that the centralizer $\mathcal{A}^\psi$ has the Haagerup property. Write $M = (N \ast \mathcal{A})^{\tau \ast \psi}$. Then $M$ is a type ${\rm II_1}$ factor and $\mathcal{F}(M) = \Gamma_{\Sp(\mathcal{A}, \psi)}$.
\end{theo}

\begin{proof}
The proof will go as the one of Theorem $\ref{rigidaraki}$. The main change here is that \emph{controlling} quasi-normalizers no longer relies on the result of Popa (namely Theorem $3.1$ in \cite{popamal1}) but on Theorem $\ref{amalga}$ of the present paper. However, we shall sketch the proof for completeness. Denote by $(\mathcal{M}, \varphi) = (N, \tau) \ast (\mathcal{A}, \psi)$, $M = \mathcal{M}^\varphi$ and $A = \mathcal{A}^\psi$. First of all, we prove that $M$ is a factor of type ${\rm II_1}$. Note that $N \subset M$. Take now $x \in \mathcal{Z}(M) = M \cap M'$. Since $N$ is diffuse, $\displaystyle{N \mathop{\nprec}_{\mathcal{M}} \C}$. Theorem $\ref{amalga}$ implies $\mathcal{M} \cap N' \subset N$. Consequently, $x \in N \cap N' = \mathcal{Z}(N) = \C1$. It follows that $\mathcal{Z}(M) = \C1$. Note that in this case, one has $\Sp(\mathcal{M}, \varphi) = \Gamma_{\Sp(\mathcal{A}, \psi)}$.

We already know that the inclusion $\Gamma_{\Sp(\mathcal{A}, \psi)} \subset \mathcal{F}(M)$ holds. Conversely, let $t \in \mathcal{F}(M)$ and let $\theta : M \to M^t$ be a $\ast$-isomorphism. We assume that $t \geq 1$. Realize $M^t := p(M_n(\C) \otimes M)p$. Let $B \subset G$ be a diffuse von Neumann subalgebra such that the inclusion $B \subset N$ is quasi-regular and has the relative property $(T)$. Since $B \subset N$ is quasi-regular, $N$ is contained in the quasi-normalizer of $B$ inside $\mathcal{M}$. Moreover, since $B$ is diffuse, Theorem $\ref{amalga}$ implies that the quasi-normalizer of $B$ inside $\mathcal{M}$ is exactly $N$. Write $Q = \theta(B)$ and $P = \theta(N)$. The inclusion $Q \subset P$ still has the relative property $(T)$ and $P$ is the quasi-normalizer of $Q$ inside $M^t$. Since $s = 1/t \leq 1$, as we did before, choose a projection $q \in Q$ with trace $s$. Write $Q^s := q Q q$ and $P^s := q P q$. We regard $Q^s \subset P^s \subset M$. The inclusion $Q^s \subset P^s$ has the relative property $(T)$, $Q^s$ is diffuse  and $P^s$ is the quasi-normalizer of $Q^s$ inside $M$. We know from Theorem $\ref{embedamalga}$ that either $\displaystyle{Q^s \mathop{\prec}_{\mathcal{M}} N}$ or $\displaystyle{Q^s \mathop{\prec}_{\mathcal{M}} \mathcal{A}}$.

\begin{claim}\label{nonembed}
$\displaystyle{Q^s \mathop{\nprec}_{\mathcal{M}} \mathcal{A}}$.
\end{claim}

\begin{proof}[Proof of Claim $\ref{nonembed}$]
Assume that $\displaystyle{Q^s \mathop{\prec}_{\mathcal{M}} \mathcal{A}}$. We shall obtain a contradiction. By Theorem $\ref{newembed}$, we know that there exist $m \geq 1$, $\gamma > 0$, a projection $e \in M_m(\C) \otimes A$, a nonzero partial isometry $v \in M_{1,m}(\C) \otimes \mathcal{M}^\gamma$ and a (unital) $\ast$-homomorphism $\rho : Q^s \to e(M_m(\C) \otimes A)e$ such that $v^*v \leq e$ and
\begin{equation*}
x v = v \rho(x), \forall x \in Q^s.
\end{equation*}
Note that $vv^* \in M \cap (Q^s)' \subset P^s$. In the same way, $v^*v \in e(M_m(\C) \otimes M)e \cap \rho(Q^s)'$. Since $\rho(Q^s)$ is diffuse, Theorem $\ref{amalga}$ (and Equation $(\ref{isoampli})$) tell us that the quasi-normalizer of $\rho(Q^s)$ inside $e (M_m(\C) \otimes \mathcal{M}) e$ is contained in $e(M_m(\C) \otimes \mathcal{A})e$. Consequently $v^*v \in e (M_m(\C) \otimes A) e$ so that we may assume $e = v^*v$. In the same way, since $e \rho(Q^s)$ is still diffuse, Theorem $\ref{amalga}$ tells us that the quasi-normalizer of $e\rho(Q^s)$ inside $e \mathcal{M} e$ is contained in $e (M_m(\C) \otimes \mathcal{A}) e$. In particular, $v^* P^s v \subset e (M_m(\C) \otimes A) e$. But the inclusion $v^* Q^s v \subset v^* P^s v$ has the relative property $(T)$ and the von Neumann algebra $v^* Q^s v$ is diffuse. Since $A$ is assumed to have the Haagerup property, this cannot happen thanks to Theorem $5.4$ in \cite{popa2001} (see also Theorem $\ref{haagfac}$ in Section $\ref{ara}$). We have a contradiction.
\end{proof}

Consequently, we obtain that $\displaystyle{Q^s \mathop{\prec}_{\mathcal{M}} N}$. We proceed as in the proof of Claim $\ref{nonembed}$ and the proof of Theorem $\ref{rigidaraki}$. Increasing $n$ if necessary, we obtain $\gamma > 0$ and a nonzero partial isometry $v \in M_{n}(\C) \otimes \mathcal{M}$, which is a $\gamma$-eigenvector for $\varphi$ such that  
\begin{equation}\label{inclusion2}
v^* Q v \subset v^* P v \subset N^{t/\gamma},
\end{equation}
with $vv^* = p$, $v^*v := q \in M_n(\C) \otimes N$, and $N^{t/\gamma} := q (M_n(\C) \otimes N) q$. We want to prove that in fact, $v^* P v = q (M_n(\C) \otimes N) q = N^{t/\gamma}$. If we do so, we are done. Indeed, we have $N \simeq N^{t/\gamma}$, and so $t/\gamma \in \mathcal{F}(N) = \{1\}$. Consequently, $t = \gamma \in \Gamma_{\Sp(\mathcal{A}, \psi)}$.

Changing $q$ to an equivalent projection in $M_n(\C) \otimes B$, we can always assume that $q \in M_n(\C) \otimes B$. Define
\begin{equation*}
Q_1 := \theta^{-1}(v (M_n(\C) \otimes B) v^*) \mbox{ and } P_1 := \theta^{-1}(v (M_n(\C) \otimes N) v^*).
\end{equation*}
The inclusion $Q_1 \subset P_1$ has the relative property $(T)$, $P_1$ is the quasi-normalizer of $Q_1$ inside $M$, and Equation $(\ref{inclusion2})$ yields $N \subset P_1$. We want to prove that $P_1 \subset N$. Once again using Theorem $\ref{embedamalga}$ and Claim $\ref{nonembed}$, we get $\displaystyle{Q_1 \mathop{\prec}_{\mathcal{M}} N}$. Thus, there exists $k \geq 1$, $\lambda > 0$ and a nonzero partial isometry $w \in M_{1, k}(\C) \otimes \mathcal{M}^\lambda$, such that $ww^* = 1$, $w^*w \in M_k(\C) \otimes N$ and $w^* P_1 w \subset N^{1/\lambda}$, where we have realized $N^{1/\lambda} := w^*w (M_k(\C) \otimes N) w^*w$. Since $ww^* = 1$, we have $P_1 w \subset w N^{1/\lambda}$.  Since $N \subset P_1$, it follows that  $N w \subset w N^{1/\lambda}$. From Theorem $\ref{amalga}$, since $N$ is diffuse, we know that any $N$-$N$ subbimodule $\mathcal{H}$ of $L^2(\mathcal{M})$ such that $\dim(\mathcal{H}_{N}) < \infty$  (as a right $N$-module) is contained in $L^2(N)$. In particular, this implies that $w \in M_{1, k}(\C) \otimes N$. Since $P_1 \subset w (N^{1/\lambda}) w^*$, we get $P_1 \subset N$. We are done.
\end{proof}

\appendix

\section{On the Polar Decomposition of a Vector}

The von Neumann algebra $\mathcal{M}$ is assumed to be endowed with a faithful normal almost periodic state $\varphi$. We regard $\mathcal{M} \subset B(L^2(\mathcal{M}, \varphi))$. For $\gamma \in \Sp(\mathcal{M}, \varphi)$, denote as usual by $\mathcal{M}^\gamma \subset \mathcal{M}$ the subspace of all $\gamma$-eigenvectors for the state $\varphi$. Denote by $\mathcal{M}_{\alg} := \span \{ \mathcal{M}^\gamma : \gamma \in \Sp(\mathcal{M}, \varphi) \}$. Let $\gamma \in \Sp(\mathcal{M}, \varphi)$. Let $\xi \in L^2(\mathcal{M}^\gamma)$ such that $\xi \neq 0$.   Let $T_\xi^0 : \widehat{\mathcal{M}_{\alg}} \to L^2(\mathcal{M}, \varphi)$ be the linear operator defined by
\begin{equation*}
T_\xi^0(\widehat{x})  =  \lambda^{-1/2} \xi x, \forall x \in \mathcal{M}^\lambda. 
\end{equation*}
The aim of this Appendix is to prove the following proposition: it is well known from specialists, but we  give a proof for the sake of completeness.

\begin{prop}\label{polar}
The densily defined operator $T_\xi^0$ is closable. Denote by $T_\xi$ its closure. The operator $T_\xi$ is affiliated with $\mathcal{M}$. Write $T_\xi = v |T_\xi|$ for its polar decomposition. Then, $v \in \mathcal{M}^\gamma$ and $|T_\xi|$ is affiliated with the centralizer $\mathcal{M}^\varphi$. Moreover, if $B \subset \mathcal{M}^\varphi$ is a von Neumann subalgebra such that for every $x \in B$, $x \xi = \xi x$, then for every $x \in B$, we have $xv = vx$.
\end{prop}

\begin{proof}
First, we prove that the operator $T_\xi^0$ is closable. It suffices to show that $(T_\xi^0)^*$ is densily defined. Let $\alpha \in \Sp(\mathcal{M}, \varphi)$. Set $\beta = \gamma \alpha$. Let $y \in \mathcal{M}^\alpha$ and $z \in \mathcal{M}^\beta$. Then,
\begin{eqnarray*}
\langle T_\xi^0(\widehat{y}), \widehat{z} \rangle & = & \alpha^{-1/2} \langle \xi y, \widehat{z} \rangle \\
& = & \alpha^{-1/2} \langle J_\varphi y^* J_\varphi \xi, \widehat{z} \rangle \\
& = & \alpha^{-1/2} \langle z^* J_\varphi y^* J_\varphi \xi, \widehat{1} \rangle \\
& = & \alpha^{-1/2} \langle J_\varphi y^* J_\varphi z^* \xi, \widehat{1} \rangle \\
& = & \alpha^{-1/2} \langle \widehat{y}, J_\varphi z^* \xi \rangle \\
& = & \alpha^{-1/2} \langle \widehat{y}, (J_\varphi \xi) z \rangle \\
& = & (\gamma \alpha)^{-1/2} \langle \widehat{y}, (S_\varphi \xi) z \rangle \\
& = & \langle \widehat{y}, T^0_{S_\varphi \xi}(\widehat{z}) \rangle.
\end{eqnarray*}
If $\beta \neq \gamma \alpha$, then $\langle T_\xi^0(\widehat{y}), \widehat{z} \rangle = \langle \widehat{y}, T^0_{S_\varphi \xi}(\widehat{z}) \rangle = 0$. Consequently, for every $y, z \in \mathcal{M}_{\alg}$, 
\begin{equation*}
\langle T_\xi^0(\widehat{y}), \widehat{z} \rangle = \langle \widehat{y}, T^0_{S_\varphi \xi}(\widehat{z}) \rangle.
\end{equation*}
Then $T_{S_\varphi \xi}^0 \subset (T_\xi^0)^*$, and so $(T_\xi^0)^*$ is densily defined. Thus $T_\xi^0$ is closable and we denote by $T_\xi$ its closure. We prove now that $T_\xi$ is affiliated with $\mathcal{M}$. 
Let $\alpha, \lambda \in \Sp(\mathcal{M}, \varphi)$. Let $a \in \mathcal{M}^\alpha$ and $x \in \mathcal{M}^\lambda$. On the one hand,
\begin{eqnarray*}
T_\xi^0 J_\varphi a^* J_\varphi (\widehat{x}) & = & \alpha^{1/2} T_\xi^0 (\widehat{xa}) \\
& = & \lambda^{-1/2} \xi xa.
\end{eqnarray*}
On the other hand,
\begin{equation*}
J_\varphi a^*J_\varphi T_\xi^0(\widehat{x}) = \lambda^{-1/2} \xi xa.
\end{equation*}
Consequently, we have $J_\varphi a^*J_\varphi T_\xi^0 \subset T_\xi^0 J_\varphi a^* J_\varphi$, for every $a \in \mathcal{M}^\lambda$. Since $J_\varphi \mathcal{M}_{\alg} J_\varphi$ is $\sigma$-weakly dense in $\mathcal{M}'$, it follows that $T_\xi$ is affiliated with $\mathcal{M}$.

Write $T_\xi = v |T_\xi|$ for the polar decomposition of $T_\xi$. We know that $v \in \mathcal{M}$. Let $\lambda \in \Sp(\mathcal{M}, \varphi)$ and $x \in \mathcal{M}^\lambda$. Then, for every $t \in \R$,
\begin{eqnarray*}
\Delta_\varphi^{it} T_\xi^0 \Delta_\varphi^{-it} (\widehat{x}) & = & \lambda^{-it} \Delta_\varphi^{it} T_\xi^0 (\widehat{x}) \\
& = & \lambda^{-1/2} \lambda^{-it} \Delta_\varphi^{it} (\xi x) \\
& = & \gamma^{it} \lambda^{-1/2}  \xi x \\
& = & \gamma^{it}  T^0_\xi (\widehat{x}).
\end{eqnarray*}
Thus, it follows that $\Delta_\varphi^{it} T_\xi \Delta_\varphi^{-it} = \gamma^{it} T_\xi$, for any $t \in \R$. But, we also have 
\begin{equation*}
\Delta_\varphi^{it} T_\xi \Delta_\varphi^{-it} = (\Delta_\varphi^{it} v \Delta_\varphi^{-it}) (\Delta_\varphi^{it} |T_\xi| \Delta_\varphi^{-it}).
\end{equation*}
By uniqueness of the polar decomposition, we get for every $t \in \R$, 
\begin{eqnarray*}
\Delta_\varphi^{it} v \Delta_\varphi^{-it} & = & \gamma^{it} v, \\
\Delta_\varphi^{it} |T_\xi| \Delta_\varphi^{-it} & = & |T_\xi|.
\end{eqnarray*}
Consequently, $\sigma_t^\varphi (v) = \gamma^{it} v$, for every $t \in \R$, and so $v \in \mathcal{M}^\gamma$. Since $\mathcal{M}^\varphi = \mathcal{M} \cap \{ \Delta_\varphi^{it} : t \in \R \}'$, it follows that $|T_\xi|$ is affiliated with $\mathcal{M}^\varphi$. 

At last, let $B \subset \mathcal{M}^\varphi$ be a von Neumann subalgebra such that for any $x \in B$, $x \xi = \xi x$. Fix $x \in B$. It is straightforward to check that $x T_\xi \subset T_\xi x$. We also have $x (T_\xi)^* \subset (T_\xi)^* x$, and so $x(T_\xi)^*T_\xi \subset (T_\xi)^*T_\xi x$. By functional calculus, it follows that $x |T_\xi| \subset |T_\xi| x$. Moreover, since $\mathcal{M}^\varphi$ is a finite von Neumann algebra, since $x \in \mathcal{M}^\varphi$ and $|T_\xi|$ is affiliated with $\mathcal{M}^\varphi$, it follows that $x |T_\xi|$ and $|T_\xi| x$ are closed, affiliated with $\mathcal{M}^\varphi$ and consequently the equality $x |T_\xi| = |T_\xi| x$ holds. Thus,
\begin{eqnarray*}
x v |T_\xi| & = & x T_\xi \\
& \subset & T_\xi x \\
& \subset & v |T_\xi| x \\
& \subset & v x |T_\xi|. 
\end{eqnarray*}
It follows that $xv$ and $vx$ coincide on the range of $|T_\xi|$, and so $xv = vx$. Thus, $xv = vx$, for every $x \in B$.
\end{proof}


\bibliographystyle{plain}

\begin{thebibliography}{AA}

\bibitem{barnett95} {\sc L. Barnett}, {\it Free product von Neumann algebras of type ${\rm III}$}. Proc. Amer. Math. Soc. {\bf 123} (1995), 543--553.

\bibitem{burger} {\sc M. Burger}, {\it Kazhdan constants for $\SL_3(\Z)$}. J. Reine Angew. Math. {\bf 413} (1991), 36--67.


\bibitem{haagerup} {\sc P.A. Cherix, M. Cowling, P. Jolissaint, P. Julg \& A. Valette}, {\it Groups with the Haagerup Property.} Progress in Mathematics {\bf 197}.
Birkh\"{a}user Verlag, Basel, Boston, Berlin, 2001.

\bibitem{choda} {\sc M. Choda}, {\it Group factors of the Haagerup type.} Proc. Japan Acad. {\bf 59} (1983), 174--177.

\bibitem{CJ} {\sc A. Connes \& V.F.R. Jones}, {\it Property $(T)$ for von Neumann algebras}. Bull. London Math. Soc. {\bf 17} (1985), 57--62.

\bibitem{connes74} {\sc A. Connes}, {\it Almost periodic states and factors of type ${\rm III_1}$}. J. Funct. Anal. {\bf 16} (1974), 415--445.

\bibitem{connes73} {\sc A. Connes},
{\it Une classification des facteurs de type {\rm III}.} Ann. Sci. {\'E}cole Norm. Sup. {\bf 6} (1973), 133--252.

\bibitem{dykema96} {\sc K. Dykema},
{ \it Free products of finite-dimensional and other von Neumann algebras
with respect to non-tracial states.} Free probability theory (Waterloo, ON, 1995)
Fields Inst. Commun. {\bf 12}
Amer. Math. Soc., Providence, RI, 1997, pp. 41--88.

\bibitem{dykema94} {\sc K. Dykema}, {\it Interpolated free group factors.} Pacific J. Math. {\bf 163} (1994), 123--135.


\bibitem{fernos} {\sc T. Fern\'os},
{\it Relative Property $(T)$ and linear groups.} Ann. Inst. Fourier. {\bf 56} (2006), 1767--1804.

\bibitem{golodets} {\sc V. Y. Golodets \& N. I. Nessonov}, {\it T-property and nonisomorphic full factors of types {\rm II} and {\rm III}}. J. Funct. Anal. {\bf 70} (1987), 80--89.

\bibitem{haa} {\sc U. Haagerup}, {\it An example of non-nuclear $C^*$-algebra which has the metric approximation property.} Invent. Math. {\bf 50} (1979), 279--293.

\bibitem{HV} {\sc P. de la Harpe \& A. Valette}, {\it La propri\'et\'e $(T)$ de Kazhdan pour les groupes localement compacts.} Ast\'erisque {\bf 175}.
SMF, Paris, 1989.

\bibitem{houdayer2} {\sc C. Houdayer}, {\it On some free products of von Neumann algebras which are free Araki-Woods factors.} Int. Math. Res. Notices. Vol. {\bf 2007}, article ID rnm098, 21 pages.

\bibitem{ioana3} {\sc A. Ioana}, {\it Rigidity results for wreath product ${\rm II_1}$ factors.} J. Funct. Anal. {\bf 252} (2007), 763--791.

\bibitem{ipp} {\sc A. Ioana, J. Peterson \& S. Popa}, {\it Amalgamated free products of $w$-rigid factors and calculation of their symmetry groups.} Acta Math. {\bf 200} (2008), 85--153. 

\bibitem{jolissaint} {\sc P. Jolissaint}, {\it Haagerup approximation property for finite von Neumann algebras.} J. Operator Th. {\bf 48} (2002), 549--571.

\bibitem{kazhdan} {\sc D. Kazhdan}, {\it Connection of the dual space of a group with the structure of its subgroups.} Funct. Anal. Appl. {\bf 1} (1967), 63--65.

\bibitem{margulis} {\sc G. Margulis}, {\it Finitely-additive invariant measures on Euclidean spaces.} Ergodic Th. and Dynam. Sys. {\bf 2} (1982), 383--396.


\bibitem{popasup} {\sc S. Popa}, {\it Cocycle and orbit equivalence superrigidity for Bernoulli actions of Kazhdan groups.} Invent. math. {\bf 170} (2007), 243--295.

\bibitem{popamal1} {\sc S. Popa}, {\it Strong rigidity of ${\rm II_1}$ factors arising from malleable actions of w-rigid groups ${\rm I}$.} Invent. Math. {\bf 165} (2006), 369--408.

\bibitem{popamal2} {\sc S. Popa}, {\it Strong rigidity of ${\rm II_1}$ factors arising from malleable actions of w-rigid groups ${\rm II}$.} Invent. Math. {\bf 165} (2006), 409--453.

\bibitem{popa2001} {\sc S. Popa}, {\it On a class of type ${\rm II_1}$ factors with Betti numbers invariants.} Ann. of Math. {\bf 163} (2006), 809--899.

\bibitem{popamsri} {\sc S. Popa}, {\it Some rigidity results for non-commutative Bernoulli shifts.} J. Funct. Anal. {\bf 230} (2006), 273--328.

\bibitem{popavaes} {\sc S. Popa \& S. Vaes}, {\it Strong rigidity of generalized Bernoulli actions and computations of their symmetry groups.} Adv. Math. {\bf 217} (2008), 833--872.

\bibitem{popavaes2} {\sc S. Popa \& S. Vaes}, {\it Actions of $\mathbf{F}_\infty$ whose ${\rm II_1}$ factors and orbit equivalence relations have prescribed fundamental group.} arXiv:0803.3351. 

\bibitem{radulescu1994} {\sc F. R\u{a}dulescu}, { \it Random matrices, amalgamated free products and subfactors of the von Neumann algebra of a free group, of noninteger index.} Invent. Math. {\bf 115} (1994), 347--389.



\bibitem{shalom} {\sc Y. Shalom}, {\it Bounded generation and Kazhdan property $(T)$}. Publ. Math. I.H.\'E.S. {\bf 90} (1999), 145--168.


\bibitem{shlya2004} {\sc D. Shlyakhtenko}, {\it On the classification of full
  factors of type {\rm III}.} Trans. Amer. Math. Soc. {\bf 356} (2004), 4143--4159.



\bibitem{shlya99} {\sc D. Shlyakhtenko}, {\it $A$-valued semicircular systems.} J. Funct. Anal. {\bf 166} (1999), 1--47.

\bibitem{shlya98} {\sc D. Shlyakhtenko}, {\it Some applications of freeness with amalgamation.} J. Reine Angew. Math. {\bf 500} (1998), 191--212.

\bibitem{shlya97} {\sc D. Shlyakhtenko}, {\it Free quasi-free states.} Pacific J. Math. {\bf 177} (1997), 329--368.

\bibitem{takesakiII} {\sc M. Takesaki}, { \it Theory of Operator Algebras ${\rm II}$.} EMS {\bf 125}. Springer-Verlag, Berlin, Heidelberg, New-York, 2000.

\bibitem{vaesbern} {\sc S. Vaes}, {\it Rigidity results for Bernoulli actions and their von Neumann algebras (after S. Popa).} S{\'e}minaire Bourbaki, expos\'e 961. Ast\'erisque {\bf 311} (2007), 237-294.

\bibitem{vaes2004} {\sc S. Vaes}, {\it \'Etats quasi-libres libres et facteurs de
  type {\rm III} (d'apr{\`e}s D. Shlyakhtenko).} S{\'e}minaire Bourbaki, expos\'e 937, {Ast\'erisque {\bf
  299}} (2005), 329--350.

\bibitem{valetterel} {\sc A. Valette},
{\it Group pairs with property $(T)$, from arithmetic lattices.} Geom. Dedicata {\bf 112} (2005), 183--196.


\bibitem{voiculescu92} {\sc D.-V. Voiculescu, K.J. Dykema \& A. Nica}, {\it Free
  random variables.} CRM Monograph Series {\bf 1}.
American Mathematical Society, Providence, RI, $1992$. 




\end{thebibliography}

\end{document}